\documentclass[oneside,a4paper,11pt,reqno]{amsart}
\usepackage{amssymb,amsmath,amsthm}

\newtheorem{hypo}{Hypothesis}

\newtheorem{prop}[hypo]{Proposition}

\newtheorem{thm}[hypo]{Theorem}

\newtheorem{lem}[hypo]{Lemma}

\newtheorem{rqe}[hypo]{Remark}

\newtheorem{coro}[hypo]{Corollary}

\DeclareMathOperator{\var}{Var}
\DeclareMathOperator{\cov}{Cov}
\DeclareMathOperator{\diam}{diam}

\def\I{\mathcal{I}}

\def\PP{\mathbb{P}}
\def\RR{\mathbb{R}}

\newcommand {\refeq}[1] {(\ref{#1})}

\title{Self-intersections of trajectories of Lorentz process}

\date\today

\author{Fran\c{c}oise P\`ene}
\address{Universit\'e de Brest,
UMR CNRS 6205, Laboratoire de Math\'ematique de Bretagne Atlantique,
6 avenue Le Gorgeu, 29238 Brest cedex, France}
\email{francoise.pene@univ-brest.fr}

\keywords{Sinai billiard, Lorentz process, self intersection \\
Fran\c{c}oise P\`ene is supported by the french ANR projects GEODE (ANR-10-JCJC-0108) and
PERTURBATIONS (ANR-10-BLAN-0106)}
\begin{document}

\begin{abstract}
We study the asymptotic behaviour of 
the number of self-intersections of a trajectory of a
$\mathbb Z^2$-periodic planar Lorentz process with strictly convex obstacles
and finite horizon. We give precise estimates for its expectation and
its variance. As a consequence, we establish the almost sure convergence
of the self-intersections with a suitable normalization.
\end{abstract}
\maketitle
\section*{Introduction}
\label{}
The {\bf Lorentz process} describes the evolution of a point particle
moving at unit speed in a domain $Q$ with elastic reflection
on $\partial Q$. We consider here a planar Lorentz process
in a $\mathbb Z^2$-periodic domain $Q\subseteq \mathbb R^2$ with strictly convex obstacles
$U_{i,\ell}$ constructed as follows. 
We choose a finite number of convex open sets $O_1,...,O_I\subset \RR^2$ with
$C^3$-smooth boundary and with non null curvature. We repeat these sets $\mathbb Z^2$-periodically by defining $U_{i,\ell}=O_i+\ell$ for every $(i,\ell)\in\{1,...,I\}\times\mathbb Z^2$. 
We suppose that the closures of the $U_{i,\ell}$ are pairwise disjoint.
Now we define the domain $Q:=\mathbb R^2\setminus\bigcup_{i=1}^I\bigcup_{\ell\in\mathbb Z^2}U_{i,\ell}$.
We assume that {\bf the horizon is finite}, which means that every line meets the boundary of $Q$ (i.e. there is no infinite free flight).
We consider a point particle moving in $Q$ with unit speed and with respect to
the Descartes reflection law at its reflection times (reflected angle=incident angle). We call configuration of a particle at
some time the couple constituted by its position
and its speed.
The Lorentz process in the domain $Q$ is the flow $(Y_t)_t$
on $Q\times S^1$ such that $Y_t$ maps the configuration at time 
$0$ to the configuration at time $t$.
We assume that the initial distribution $\mathbb P$
is uniform on $(Q\cap [0,1]^2)\times S^1$.
The study of the Lorentz process is strongly related 
to the corresponding Sinai billiard $(\bar M,\bar\mu,\bar T)$.
Recall that this billiard is the 
probability dynamical system describing the dynamics
of the Lorentz process modulo $\mathbb Z^2$
and at reflection times. 
Ergodic properties of this dynamical system have been studied namely 
by Sinai
in \cite{Sinai70} (for its ergodicity), Bunimovich and Sinai \cite{BS80,BS81}, Bunimovich, Chernov and Sinai
\cite{BCS90,BCS91} (for central limit theorems), Young \cite{Young98}
(for exponential rate of decorrelation). Other limit theorems
for the Sinai billiard
and its applications to the Lorentz process have
been investigated in many papers, let us mention namely \cite{Conze,FP00,SzV} for its ergodicity and \cite{DSzV} for some other properties.

We are interested here in the study of the following quantity, called {\bf number of self-intersections of the trajectory
of the Lorentz Process}:
$$\mathcal V_t:=\#\{(r,s)\in[0;t]^2 \ :\ \pi_Q(Y_r)=\pi_Q(Y_s)\}  ,$$
where $\pi_Q$ denotes
the canonical projection from $Q\times S^1$ to $Q$ (i.e. $\pi_Q(q,\vec v)=q$).
This quantity $\mathcal V_t$ corresponds to the number of couples of times $(r,s)$ before time $t$ such that the particle was at the same position in the plane at both times $r$ and $s$. 
We also define $V_n$ as the number of self-intersections up to the $n$th reflection time.
The studies of $\mathcal V_t$ and of $V_n$ are naturally linked.

Self-intersections of random walks have been studied by many authors
(see \cite{Chen} and references therein). Motivated by the study of planar random walks in random sceneries, Bolthausen \cite{Bol} established an exact estimate for the expectation 
of the number of self-intersections of planar recurrent random walks.
He also stated an upper bound for its variance.
This last estimate was sufficient for his purpose but not optimal. A precise estimate for this variance has recently been stated by Deligiannidis and Utev \cite{DU}.

In view of planar Lorentz process in random scenery,
another notion of self-intersections of Lorentz process arises:
the number of self-intersections of the Lorentz process seen on obstacles, i.e.
the number $\hat V_n$ of couples of times $(r,s)$ (before the $n$-th reflection)
such that the particle hit the same obstacle at both times $r$ and $s$.
This quantity has been studied in \cite{FP09b,FP13a}.
In the present work, our approach has some common points with \cite{FP09b,FP13a} but the study of $V_n$ (and thus of $\mathcal V_t$) is much more delicate than the study of $\hat V_n$ (see section \ref{toto} for some explanations).

Let us define $(\I_k,S_k)$ in $\{1,...,I\}\times\mathbb Z^2$
as the index of the obstacle hit at
the $k$-th reflection time
($(I_0,S_0)$ being the index of the obstacle at time 0
or at the last reflection time before 0). The asymptotic behaviour of $(S_n)_n$ plays
some role here. In particular, our proofs use a decorrelation
result and some precised local limit
theorems for $(S_n)_n$. As a consequence, the constants appearing
in our statements are expressed in terms 
of the asymptotic (positive) variance matrix  $\Sigma^2$
of $(k^{-1/2}S_k)_{k\ge 1}$ (with respect to $\bar\mu$).
\begin{thm}\label{thm1a}
We have
\begin{equation}\label{Vn}
{\mathbb E}_{\bar\mu}[V_n]= c n\log n+O(n),\ \ \mbox{with}\ \ 
c:=2\mathbb E_{\bar\mu}[\tau]/(\pi\sqrt{\det \Sigma^2}
    \sum_i|\partial O_i|),
\end{equation}
where $\tau$ is the free flight length until the next reflection time.
\end{thm}
\begin{thm}\label{thm1}
We have
\begin{equation}\label{Vt}
{\mathbb E}(\mathcal V_t)= \frac{2 t\log t}{\pi\sqrt{\det \Sigma^2}
    \sum_i|\partial O_i|}+O(t)\ \ \ \mbox{as }t\mbox{ goes to infinity.}
\end{equation}
\end{thm}
Let us indicate that
these results are generalized in
Corollaries \ref{coro2} and \ref{coro2c} to a wider class of initial probability measures.
\begin{thm}\label{prop1}
We have
$\var _{\bar\mu}(V_n)\sim c' n^{2}$  
with
$$c':=c^2\left(1+2J-\frac{\pi^2}6\right)\mbox{ and } 
J:=\int_{[0,1]^3}\frac{(1-(u+v+w)){\bf 1}_{\{u+v+w\le 1\}}\,  du\, dv\, dw}    {uv+uw+vw}.$$
\end{thm}
\begin{coro}\label{thm2}
The following convergences hold almost everywhere (with respect
to $\bar\mu$ and to the Lebesgue measure on $Q\times S^1$ respectively):
$$\lim_{n\rightarrow +\infty}\frac {V_n}{n\log n}=c\ \ \mbox{and}
\ \ \lim_{t\rightarrow +\infty}\frac{\mathcal V_t}{t\log t}=
  \frac 2{\pi\sqrt{\det \Sigma^2}
    \sum_i|\partial O_i|}.$$
\end{coro}
The paper is organized as follows. In Section \ref{billiard},
we introduce the billiard systems, some notations and local limit theorems with remainder terms. In Section \ref{expectation}, we prove Theorem \ref{thm1a}. 
In Section \ref{deco}, we establish a decorrelation result in view of our proof of Theorem \ref{prop1} in Section \ref{variance}.
In Section \ref{expectation2}, we use Theorems \ref{thm1a}
and \ref{prop1} to prove Theorem \ref{thm1} and some
generalization of Theorems \ref{thm1a} and \ref{thm1}
to a class of probability measures. Finally we prove Corollary \ref{thm2} in Section \ref{LFGN}.
\section{Lorentz process and billiard systems}\label{billiard}
We denote by $\langle\cdot,\cdot\rangle$ the usual scalar product on $\mathbb R^2$ and by $|\cdot|$ the supremum norm on $\mathbb R^2$.
\subsection{planar billiard system}
For any $q\in\partial Q$, we write
$\vec n_q$ for the unit normal vector to $\partial Q$ at $q$
directed into $Q$.
We consider the set $M$ of couples position-unit speed $(q,\vec v)$
corresponding to a reflected vector on $\partial Q$:
$$M:=\{(q,\vec v)\in(\partial Q)\times S^1\ :\ \langle \vec n_q,\vec v\rangle\ge 0\}.$$
For every $i\in\mathbb \{1,...,I\}$, we fix some $q_i\in\partial O_i$.
A couple $(q,\vec v)\in M$ is parametrized by $(i,\ell,r,\varphi)\in\bigcup_{i'=1}^I\{i'\}\times\mathbb Z^2\cup\frac{\mathbb R}{|\partial O_{i'}|\mathbb Z} \times \left[-\frac\pi 2,\frac \pi 2\right]$ if
\begin{itemize}
\item $q-\ell$ is the point of $\partial O_i$ with curvilinear absciss
$r$ for the trigonometric orientation (starting from $q_i$)
\item $\varphi$ is the angular measure of $\widehat{(\vec n_q,\vec v)}$.
\end{itemize}
We consider the transformation $T$ mapping a reflected vector to the
reflected vector corresponding to the next collision time.
$T$ preserves
the (infinite) measure $\mu$ with density $\cos(\varphi)$ with respect
to the measure $drd\varphi$ on $M$. 
This infinite measure dynamical system $(M,\mu,T)$ is called {\bf planar billiard system}.
We endow $M$ with a metric $d$ equal to 
$\max(|r-r'|,|\varphi-\varphi'|)$ on any obstacle $\partial U_{i,\ell}$.
We define the map $\tau:M\rightarrow[0,+\infty[ $ by
$$\tau(q,\vec v):=\min\{s>0\ :\ q+s\vec v\in\partial Q\} ,$$
which corresponds to the length of the free flight of 
a particle starting from $q$ with initial speed $\vec v$.
Due to our assumptions, we have
$$\min\tau>0\ \ \  \mbox{and} \ \ \ \max\tau<\infty.$$
We define $R_0$ as the set of $(q,\vec v)\in M$ with
$\vec v$ tangent to $\partial Q$ at $q$ (this set corresponds to
$\{\varphi=0\}$).
For any integers $k\le \ell$, we write
$R_{k,\ell}=\bigcap_{m=k}^\ell T^{m}(R_0)$
and $\xi_{k}^\ell$ for the set of connected components of
$M\setminus R_{-\ell,-k}$. Due to the hyperbolic properties of $T$,
it is easy to see that (see for example \cite[Lemma A.1]{FPBS09})
\begin{equation}\label{xi}
\exists c_0>0,\ \exists \delta\in(0,1),\ \ \forall k\ge 1,\ \ 
\forall \mathcal C\in\xi_{-k}^k,\ \ \ \diam(\mathcal C)\le c_0\delta^k.
\end{equation}
We recall that $T$ is discontinuous but $\frac 12$-H\"older continuous
on each connected component of $M\setminus R_{-1,0}$.
\subsection{Lorentz process}
To avoid ambiguity, at collision times, we only consider reflected vectors.
The set of configurations is then
$$\mathcal M:=((Q\setminus \partial Q)\times S^1)\cup M\subseteq Q\times
S^1.$$
The Lorentz process is the flow $(Y_t)_t$ defined on $\mathcal M$ such that,
for every $(q,\vec v)\in\mathcal M$, $Y_t(q,\vec v)=(q_t,\vec v_t)$ is the couple
position-speed at time $t$ of a particle that was at position $q$ with speed $\vec v$
at time $0$. This flow preserves the measure $\nu$ on $\mathcal M$,
where $\nu$ is the product of the Lebesgue measure on $Q$ and
of the uniform measure on $S^1$.

This flow is naturally identified with the suspension flow $(\tilde Y_t)_t$ over $(M,\mu,T)$
with roof function $\tau$. Indeed, we recall that $(\tilde Y_t)_t$ is defined by $\tilde Y_t(x,s)=(x,s+t)$
on the set
$$\tilde{\mathcal M}:=\{(x,s)\in M\times[0,+\infty[\ :\ s\le\tau(x)\},\ \ \mbox{with the identifications}\ \ 
(x,\tau(x))\equiv(T(x),0).$$
The flow $(\tilde Y_t)_t$ preserves the measure $\tilde\nu$ on $\tilde{\mathcal M}$ given by $d\tilde\nu(x,s)=d\mu(x)ds$.
Now, we define $\Delta:\tilde{\mathcal M}\rightarrow\mathcal M$ by
$$\Delta((q,\vec v),s)=(q+s\vec v,\vec v) \ \ \ \mbox{if}\ \ s<\tau(q,\vec v).$$
We have
\begin{equation}\label{suspension}
Y_t=\Delta\circ\tilde Y_t\circ\Delta^{-1}\ \ \mbox{and}\ \ \Delta_*(\tilde\nu)=\nu. 
\end{equation}
\subsection{Billiard system with finite measure}
We define $\bar M$ as the set of $(q,\vec v)\in M$ such that
$q\in\bigcup_{i=1}^I\partial O_i$. A point of $\bar M$
is now parametrized by $(i,r,\varphi)$.
We consider the transformation $\bar T:\bar M\rightarrow\bar M$,
corresponding to $T$ modulo $\mathbb Z^2$. More precisely,
if $T(q,\vec v)=(q',\vec v')$, then $\bar T(q,\vec v)=(q",\vec v)$
with $q"\in (q'+\mathbb Z^2)\cap\cup_{i=1}^I\partial O_i$.
This transformation $\bar T$ preserves
the probability measure $\bar\mu$ of density $\cos(\varphi)/
(2\sum_i|\partial O_i|)$ with respect to $drd\varphi$.

We call {\bf toral billiard system} the probability dynamical system $(\bar M,\bar\mu,\bar T)$.

It is easy to see that $(M,\mu,T)$ corresponds to the cylindrical extension of
$(\bar M,\bar\mu,\bar T)$ by $\Psi:\bar M\rightarrow\mathbb Z^2$ given by $\Psi=(S_1)_{|\bar M}$ (with $S_n$ defined in the introduction).
Indeed
$$\forall((q,\vec v),\ell)\in\bar M\times\mathbb Z^2,\ \ \ 
T(q+\ell,\vec v)=(q'+\ell+\Psi(q,\vec v),\vec v')\ \ \mbox{if}\ \ (q',\vec v')=\bar T(q,\vec v).$$
More generally we have
\begin{equation}\label{extension}
\forall((q,\vec v),\ell)\in\bar M\times\mathbb Z^2,\ \forall n\ge 1,\ \  
T^n(q+\ell,\vec v)=(q_n+\ell+\sum_{k=0}^{n-1}\Psi(\bar T^k(q,\vec v)),\vec v_n)\ \ \mbox{if}\ \ (q_n,\vec v_n)=\bar T^n(q,\vec v).
\end{equation}
Observe that $\sum_{k=0}^{n-1}\Psi\circ\bar T^k=S_n$ on $\bar M$.
We recall the following local limit theorem with remainder term.
We set $\beta:= \frac 1{2\pi\sqrt{\det \Sigma^2}}$.
\begin{prop}[Proposition 4.1 of \cite{FPBS09}]
Let $p>1$. There exists $c>0$ such that, for any $k\ge 1$,
if $A\subseteq \bar M$ is a union of components of $\xi_{-k}^k$
and $B\subseteq \bar M$ is a union of
components $\xi_{-k}^\infty$, then for any $n>2k$ and $N\in\mathbb Z^2$
$$\left|\bar\mu(A\cap\{S_n=N\}\cap\bar T^{-n}(B))-\frac{\beta
   e^{-\frac 1{2(n-2k)}\langle (\Sigma^2)^{-1}N,N\rangle }}
     {n-2k}\bar\mu(A)\bar\mu(B)\right|\le \frac{ck\bar\mu
   (B)^{\frac 1p}}{(n-2k)^{\frac 32}}. $$ 
\end{prop}
Note that, if we suppose $n\ge 3k$, we can replace the conclusion 
of this result by
\begin{equation}\label{eqprop4}
\left|\bar\mu(A\cap\{S_n=N\}\cap\bar T^{-n}(B))-\frac{\beta
   e^{-\frac 1{2n}\langle (\Sigma^2)^{-1}N,N\rangle }}
     {n}\bar\mu(A)\bar\mu(B)\right|\le \frac{ck\bar\mu
   (B)^{\frac 1p}}{n^{\frac 32}}.
\end{equation}
\begin{rqe}
Observe that, since the billiard system $(\bar M,\bar\mu,\bar T)$ is time reversible, 
if $A\subseteq \bar M$ is a union of components of $\xi_{-\infty}^k$
and $B\subseteq \bar M$ is a union of
components $\xi_{-k}^k$, if $n>3k$ then we have
\begin{equation}\label{eqprop4bis}
\left|\bar\mu(A\cap\{S_n=N\}\cap\bar T^{-n}(B))-\frac{\beta
   e^{-\frac 1{2n}\langle (\Sigma^2)^{-1}N,N\rangle }}
     {n}\bar\mu(A)\bar\mu(B)\right|\le \frac{ck\bar\mu
   (A)^{\frac 1p}}{n^{\frac 32}}.
\end{equation}
\end{rqe}
Estimates \refeq{eqprop4} and \refeq{eqprop4bis} will be enough
most of the time but not every time. We will also
use the following refinements of the local limit theorem.
\begin{prop}[Proposition 4 of \cite{FP09b}]\label{prop4a}
Let any real number $p>1$. There exist  $a_0>0$ and $K_1>0$ such that,
for any integers $k\ge 0$, $n\ge 1$, any measurable set $A\subseteq \bar M$ union of elements of $\xi_{0}^k$,
any measurable set $B\subseteq\bar M$
union of elements of $\xi_{0}^{+\infty}$, 
for any $N\in\mathbb Z^2$, we have
\begin{multline*}
\left|\bar\mu(A\cap\{S_{n+k}-S_k=N\}\cap\bar T^{-(n+k)}B)
       -\frac{\beta\bar\mu(A)\bar\mu(B)}{ n}
       e^{-\frac 1{2n}\langle(\Sigma^2)^{-1}N,N\rangle}\right|\\
\le K_1\left(\frac{\bar\mu(B)+\bar\mu(A)\bar\mu(B)^{\frac 1p}}{n^{\frac 32}}\left(\frac{|N|}{\sqrt{n}}+\frac{|N|^3}{n^{\frac 32}}\right)
    e^{-\frac {a_0}{n}|N|^2} +
    \frac{\bar\mu(B)^{\frac 1 p}}{n^2} \right).
\end{multline*}
\end{prop}
We generalize this result as follows.
\begin{prop}\label{prop4}
Let any real number $p>1$. There exist $C>0$, $a_0>0$ and $K_1>0$ such that,
for any integers $k\ge 0$, $n\ge 1$ such that $n\ge 4k$, any measurable set $A\subseteq \bar M$ union of elements of $\xi_{-k}^k$,
any measurable set $B\subseteq\bar M$
union of elements of $\xi_{-k}^{+\infty}$, 
for any $N\in\mathbb Z^2$, we have
\begin{multline*}
\left|\bar\mu(A\cap\{S_{n}=N\}\cap\bar T^{-n}B)
       -\frac{\beta\bar\mu(A)\bar\mu(B)}{ n}
       e^{-\frac 1{2n}\langle(\Sigma^2)^{-1}N,N\rangle}\right|\\
\le K_1\, k\left(\frac{\bar\mu(B)+\bar\mu(A)\bar\mu(B)^{\frac 1p}}{n^{\frac 32}}\left(\frac{|N|}{\sqrt{n}}+\frac{|N|^3}{n^{\frac 32}}\right)
   e^{-\frac {a_0}{n}(\max(|N|-2k,0))^2} +
   k \frac{\bar\mu(B)^{\frac 1 p}}{n^2} \right).
\end{multline*}
\end{prop}
\begin{proof}
Observe that $\bar T^{-k} A$ is a union of elements of
$\xi_0^{2k}$ and that $\bar T^{-k} B$ is a union of elements of
$\xi_0^{+\infty}$.
We have
\begin{eqnarray*}
\bar\mu(A\cap\{S_{n}=N\}\cap\bar T^{-n}B)
&=&\bar\mu(\bar T^{-k}A\cap\{S_{n+k}-S_k=N\}\cap\bar T^{-(n+k)}B)\\
&=&\sum_{x',y'}\bar\mu(A_{x'};S_{n}-S_{2k}=N-x'-y';\bar T^{-n}B_{y'}),
\end{eqnarray*}
with $A_{x'}:=\bar T^{-k}A\cap\{S_{2k}-S_k=x'\}$ and $B_{y'}:=
\bar T^{-k}B\cap\{S_k=y'\}$ and where the sum is taken over
$x',y'\in\mathbb Z^2$ such that $|x'|\le k\Vert S_1\Vert_\infty$
and  $|y'|\le k\Vert S_1\Vert_\infty$;
Applying Proposition \ref{prop4a} with $(A_{x'},B_{y'})$
and using the fact that $n-2k\ge n/2$, we obtain the
result.
\end{proof}
\begin{rqe}\label{remprop4}
Observe again that, by time reversibility, if $A$ is a union of
elements of $\xi_{-\infty}^k$,
if $B$ is a union of components of $\xi_{-k}^k$ and
if $n\ge 4k$, then 
\begin{multline*}
\left|\bar\mu(A\cap\{S_{n}=N\}\cap\bar T^{-n}B)
       -\frac{\beta\bar\mu(A)\bar\mu(B)}{ n}
       e^{-\frac 1{2n}\langle(\Sigma^2)^{-1}N,N\rangle}\right|\\
\le K_1\, k\left(\frac{\bar\mu(A)+\bar\mu(B)\bar\mu(A)^{\frac 1p}}{n^{\frac 32}}\left(\frac{|N|}{\sqrt{n}}+\frac{|N|^3}{n^{\frac 32}}\right)
   e^{-\frac {a_0}{n}(\max(|N|-2k,0))^2} +
   k \frac{\bar\mu(A)^{\frac 1 p}}{n^2} \right).
\end{multline*}
\end{rqe}

%
%
%
%
%
\section{Proof of Theorem \ref{thm1a}}\label{expectation}\label{toto}
Observe that the trajectory of the particle (starting from $M$) up to the $n$-th
reflection is $\bigcup_{j=0}^{n-1}[\pi_Q\circ T^j,\pi_Q\circ T^{j+1}]$.
So we have $\bar\mu$-almost surely
$$V_n=\sum_{k,j=0}^{n-1}{\mathbf 1}_{E_{k,j}}
=n+2\sum_{k=1}^{n-1}\sum_{j=0}^{n-1-k}{\mathbf 1}_{E_{j,j+k}},$$
with 
$$E_{j,k}:=\{[\pi_Q\circ T^j,\pi_Q\circ T^{j+1}]
     \cap [\pi_Q\circ T^k,\pi_Q\circ T^{k+1}] \ne\emptyset\}.$$
Hence
\begin{equation}\label{EVn}
 {\bar\mu}(V_n)=n+2\sum_{k=1}^n
     (n-k)\bar\mu(E_{0,k}).
\end{equation}
\begin{prop}\label{prop2}
There exists $\eta>0$ such that
$ \bar\mu(E_{0,n})= c/(2n)+O(n^{-1-\eta})$, with $c$ defined
in \refeq{Vn}.
\end{prop}
\begin{proof}[Proof of Theorem \ref{thm1a}]
It follows directly from \refeq{EVn} and from Proposition \ref{prop2}.
Indeed
$$\sum_{k=1}^n\frac{n-k}k=n(\sum_{k=1}^nk^{-1})-n=n\log(n)+O(n) $$
and
$$\sum_{k=1}^n\frac{n-k}{k^{1+\eta}}=n(\sum_{k=1}^nk^{-1-\eta})-
\sum_{k=1}^nk^{-\eta}=O(n).$$
\end{proof}
Before going into the proof of Proposition \ref{prop2}, 
let us see the common points between $\hat V_n$ and $V_n$ and let us
also explain why the study of $V_n$ requires more
subtle estimates than the study of $\hat V_n$.
Recall that $\hat V_n=\sum_{k,j=1}^{n}{\mathbf 1}_{(\mathcal I_k,S_k)=
(I_j,S_j)}$. So
${\mathbb E}_{\bar\mu}[\hat V_n]=n+2\sum_{k=1}^{n-1}
     (n-k)\bar\mu(\hat E_{0,k}),$
with $\hat E_{0,k}=\bigcup_{i=1}^I\{I_0=i,\ S_k=0,\ \I_k=i\}$.
This expression may appear similar to \refeq{EVn},
but $E_{0,k}$ is more complicate than $\hat E_{0,k}$. Indeed,
in $\bar M$, we have
\begin{eqnarray*}
E_{0,k}&=&\bigcup_{x\in\bar M}(\{x\}\cap T^{-k}(V^{(x)}))\\
&=&\bigcup_{N\in\mathbb Z^2}
\bigcup_{x\in\bar M}(\{x\}\cap\{S_k=N\}\cap\bar T^{-k}(\bar M\cap(V^{(x)}-N))),
\ \ \mbox{due to \refeq{extension}}
\end{eqnarray*}
with
$$V^{(x)}:=\{y\in M\ :\ [\pi_Q(y),\pi_Q\circ T(y)]
\cap[\pi_Q(x),\pi_Q\circ T(x)]\ne \emptyset\}$$
and with $A-N=\{(q-N,\vec v)\ :\ (q,\vec v)\in A\}$.
The union on $N$ is not a problem (it is a finite union since the 
horizon is finite), the main problem is that the union on $x$
is not finite. Indeed the set $V^{(x)}$ depends on $x$ (and not only on the obstacle
containing $x$).
\begin{lem}
We have
$\bar\mu(V^{(x)}+\mathbb Z^2)=\frac{2\tau(x)}{\sum_i|\partial O_i|}$.
\end{lem}
\begin{proof}
We use the fact that the measure $\cos\varphi\, drd\varphi$ is preserved
by billiard maps. So, adding the virtual obstacle
$[\pi_Q(x),\pi_Q(T(x))]$, we obtain that $\mu(V^{(x)})$ is equal to the measure
of the set of vectors based on $[\pi_Q(x),\pi_Q\circ T(x)]$
for the measure $|\cos\varphi|\, drd\varphi$, which is equal to
$4\tau(x)$ (since $\tau(x)$ is the length of $[\pi_Q(x),\pi_Q(T(x))]$).
\end{proof}
\begin{proof}[Proof of Proposition \ref{prop2}]
There exists $C>0$ such that, for
any $\varepsilon>0$, any integer $n\ge 1$, any $x_0\in \bar M$,
any connected component $\mathcal C$ of
$B(x_0,\varepsilon)\setminus R_{-1,0}$ 
and any $x\in \mathcal C$, we have
$$(\mathcal C\cap E_{0,n})\triangle(\mathcal C\cap
     T^{-n}V^{(x)}) \subseteq \mathcal C \cap T^{-n}
\mathcal D_{\mathcal C},$$
with 
$$D_{\mathcal C}:=\pi_Q^{-1}\pi_Q(\mathcal C)\cup (\pi_Q^{-1}\pi_Q(T(\mathcal C))
\cup T^{-1}(\pi_Q^{-1}\pi_Q(\mathcal C))\cup T^{-1}(\pi_Q^{-1}\pi_Q(T(\mathcal C)))
\subseteq\mathcal E_{x,\varepsilon}$$
 and
$$\mathcal E_{x,\varepsilon}:=
\pi_Q^{-1}\pi_Q(B(x,\varepsilon)\cup B(T(x),C\sqrt{\varepsilon}) )
\cup T^{-1}(\pi_Q^{-1}\pi_Q(B(x,\varepsilon)\cup B(T(x),C\sqrt{\varepsilon}))),$$
since $T$ is $\frac 12$-H\"older continuous on each connected component
of $ M\setminus R_{-1,0}$.
Take $(\varepsilon,k)$ such that
$\varepsilon^2=n^{-\frac 1{10}}=\delta^k$ (with $\delta$ of (\ref{xi})).
For any connected component $\mathcal C$ of 
$B(x_0,\varepsilon)\setminus R_{-1,0}$, we choose (in a measurable way) 
a point  $x=x_{\mathcal C}\in \mathcal C$ and
define
\begin{equation}\label{EnC}
\tilde E_{n,\mathcal C}:=\tilde{\mathcal C}\cap T^{-n}\tilde V^{(x)},\ \ \ \mbox{with}\ \ \tilde{\mathcal C}:=
\bigcup_{Z\in\xi_{-k}^k:Z\cap \mathcal C\ne\emptyset}Z\ \ \mbox{and}\ \ 
\tilde V^{(x)}:=\bigcup_{Z\in\xi_{-k}^{k}:Z\cap 
   V^{(x)}\ne\emptyset}Z.
\end{equation}
We have
$|\bar\mu(\mathcal C\cap E_{0,n})-\bar\mu(\tilde E_{n,\mathcal C}) |
  \le\bar\mu(\tilde{\mathcal D}_{n,\mathcal C}),$
with
\begin{equation}\label{DnC}
 \tilde{\mathcal D}_{n,\mathcal C}:=\tilde {\mathcal C}\cap T^{-n}\tilde
{\mathcal D}_{\mathcal C},\ \ \mbox{with}\ \ 
\tilde
{\mathcal D}_{\mathcal C}:=\bigcup_{Z\in\xi_{-k}^k:Z\cap 
   {\mathcal D}_{\mathcal C}\ne\emptyset}Z.
\end{equation}
Observe that 
$\pi_Q(\bigcup_{x\in\bar M}V^{(x)})$ is contained in 
$\bigcup_{i=1}^I\bigcup_{|\ell|\le \Vert S_1\Vert_\infty}
(\partial O_i+\ell)$.
Therefore, due to \refeq{extension}
\begin{equation}\label{tildeE}
\tilde E_{n,\mathcal C}=\bigcup_{|\ell|\le \Vert S_1\Vert_\infty}(\tilde{\mathcal C}
    \cap\{S_n=\ell\}\cap\bar T^{-n}(\bar M\cap(\tilde V^{(x)}-\ell)))
\end{equation}
and
\begin{equation}\label{tildeD}
\tilde{\mathcal D}_{n,\mathcal C}=
\bigcup_{|\ell|\le \Vert S_1\Vert_\infty}(\tilde{\mathcal C}
    \cap\{S_n=\ell\}\cap\bar T^{-n}(\bar M\cap(\tilde {\mathcal D}_{\mathcal C}-\ell))).
\end{equation}
Let $p\in(1,2)$.
Due to \refeq{eqprop4bis} and (\ref{xi}), we conclude that
there exist $\tilde C,\tilde C_0,\tilde C_1>0$ such that, 
for any $\varepsilon>0$, any integer $n\ge 1$, any $x_0\in M$,
any connected component $\mathcal C$ of
$B(x_0,\varepsilon)\setminus R_{-1,0}$ 
and any $x\in \mathcal C$, we have
$$
|\bar\mu(\mathcal C\cap E_{0,n})-\bar\mu(\tilde E_{n,\mathcal C}) |
\le \bar\mu(\tilde{\mathcal D}_{n,\mathcal C})
\le \tilde C \left(\frac{\varepsilon^{2} 
\delta^k}{n}+\frac{k\varepsilon^{\frac 2 p}}{n^{\frac 32}}
  \right)\le\tilde C_0\frac{\varepsilon^{2} 
\delta^k}{n}
$$
and
\begin{eqnarray*}
\bar\mu(\mathcal C\cap E_{0,n})
&=& \pm\bar\mu(\tilde{\mathcal D}_{n,\mathcal C})+ 
          \bar\mu(\tilde E_{n,\mathcal C})\\
&=& \pm\tilde c\left(\frac{\varepsilon^{2}\delta^k}n
  +\frac {k\varepsilon^{\frac 2p}}
  {n^{\frac 32}}\right)
  +\frac{2\beta\bar\mu(\mathcal C)\bar\mu(V^{(x)})(1\pm\delta^{\frac k 2})}{n}\\
&=& \pm 2\tilde c n^{-\frac{23}{20}}+\frac{2\beta\bar\mu(\mathcal C)\tau(x)}{n\sum_i|\partial O_i|}.
\end{eqnarray*}
Let $m\ge 1$.
We consider a $\bar\mu$-essential partition of $\bar M$ in rectangles $(P_m^{(i,j,\ell)})_{i\in\{1,...,I\},j,\ell\in\{0,...,m-1\}}$ given by
$$P_m^{(i,j,\ell)}:=\left\{(i,\bar r,\varphi):r\in \left[\frac{j|\partial O_i|}m;
\frac{(j+1)|\partial O_i|}m\right],\ 
   \varphi\in\left[-\frac\pi 2+\frac{\ell\pi}m;
     -\frac\pi 2+\frac{(\ell+1)\pi}m \right]\right\}.$$
We write $\mathcal P_m$ for the union on $(i,j,\ell)$ of the 
partition of $P_m^{(i,j,\ell)}\setminus R_{-1,0}$
in connected components. 
Taking $\varepsilon^{-1}=m=n^{1/20}$  and $k$ such that
$\delta^k=n^{-1/10}$.
We obtain
\begin{eqnarray*}
\bar\mu(E_{0,n})&=&\sum_{\mathcal C\in\mathcal P_m}
    \bar\mu(\mathcal C\cap E_{0,n})\\
&=& \pm n^{-\frac {21}{20}}+\sum_{\mathcal C\in\mathcal P_m}
   \frac{2\beta{\mathbb E}_{\bar\mu}[\tau{\mathbf 1}_{\mathcal C}]}
{n\sum_i|\partial O_i|}\\
&=&\frac{2\beta{\mathbb E}_{\bar\mu}[\tau]}{n\sum_i|\partial O_i|}+O(n^{-\frac{21}{20}}),
\end{eqnarray*}
using the fact that $\tau$ is $1/2$-H\"older continuous on each
connected component of $\bar M\setminus R_{-1,0}$.
\end{proof}
\section{A decorrelation result}\label{deco}
Let us recall some
facts on the towers constructed by Young \cite{Young98}.
These towers are two dynamical systems $(\tilde M,\tilde\mu,\tilde T)$
and $(\hat M,\hat\mu,\hat T)$ such that $(\tilde M,\tilde\mu,\tilde T)$
is an extension of $(\hat M,\hat\mu,\hat T)$ and 
$(\bar M,\bar\mu,\bar T)$. This means that there exist two measurable
maps $\tilde\pi:\tilde M\rightarrow\bar M$ and
$\hat\pi:\tilde M\rightarrow\hat M$ such that: $\tilde\pi\circ\tilde T=\bar T\circ\tilde\pi$, $\hat\pi\circ\tilde T=\hat T\circ\hat\pi$,
$\bar\mu=(\tilde\pi)_*\tilde\mu$ and $\hat\mu=(\hat\pi)_*\tilde\mu$.
Young defines a separation time $\hat s$ on $\hat M$ such that
if $\hat s(x,y)\ge n$, we have $\hat s(x,y)=n+\hat s(\hat T^nx,
\hat T^ny)$ and $\tilde\pi\hat\pi^{-1}(\{x\})$,
$\tilde\pi\hat\pi^{-1}(\{y\})$ are contained in the same atom of
$\xi_{0}^n$.
For any $\beta_0\in(0,1)$ and any $\varepsilon_0\ge 0$, Young defines
a Banach space $(\mathcal V_{\beta_0,\varepsilon_0},\Vert  \cdot\Vert _{{(\beta_0,\varepsilon_0)}})$ containing $\mathbf 1_{\hat M}$.
Let $p$ be fixed and set $q:=p/(p-1)$. It is possible to find $\beta_0\in(0,1)$ and $\varepsilon_0>0$ such that 
\begin{equation}\label{normeq}
\Vert \cdot\Vert _{L^{q}(\hat\mu)}\le C_0\Vert \cdot\Vert _{{(\beta_0,\varepsilon_0)}} ,\ \ \mbox{for some}\  C_0>0.
\end{equation}
From now on, we write $(\mathcal V,\Vert\cdot\Vert)=
(\mathcal V_{\beta_0,\varepsilon_0},\Vert  \cdot\Vert _{{(\beta_0,\varepsilon_0)}})$ for this choice of $(\beta_0,\varepsilon_0)$.
Lemma 10 of \cite{FP09b} states that
\begin{equation}\label{lem10}
\Vert  gh\Vert \le \Vert  g\Vert _{(\beta_0,0)}\Vert  h\Vert .
\end{equation}
We recall that, due to Young's construction, 
if $f$ is constant on each element of $\xi_0^N$, then there exists
a measurable $\hat f$ defined on $\hat M$ such that
\begin{equation}\label{norm0}
f\circ\tilde\pi=\hat f\circ\hat \pi\ \ \mbox{with}\ \ 
\Vert \hat f\Vert _{(\beta_0,0)}\le \Vert f\Vert _\infty(1+2\beta_0^{-N}).
\end{equation}
Let $P$ be the transfer operator on $L^q$ of $f\mapsto f\circ\hat T$
seen as an operator on $L^p$. Young proved the quasicompacity of this operator $P$ on $\mathcal V$. As in \cite{FP09b}, 
we consider here an adaptation of the 
construction of Young's towers such that 1 is the only dominating eigenvalue of $P$ on $\mathcal V$ and has multiplicity one.
Hence, there exist $K_0>0$ and $a>0$ such that
\begin{equation}\label{expo}
\forall n\ge 1,\ \ \Vert  P^n(\cdot)-{\mathbb E}_{\hat\mu}[\cdot]\Vert \le K_0e^{-an}.
\end{equation}
Thanks to this property, Young established an exponential rate of decorrelation. 
Let us consider $\Psi:\bar M\rightarrow\mathbb Z^2$ 
the cell-shift function. Recall that, on $\bar M$, 
$S_n=\sum_{k=0}^{n-1}\Psi\circ\bar T^k$.
Since $\Psi$ is constant on each element of $\xi_0^{1}$, there
exists $\hat\Psi:\hat M\rightarrow\mathbb Z^2$ such that
$\hat\Psi\circ\hat\pi=\Psi\circ\pi$
and the coordinates of $\hat\Psi$ are in
$\mathcal V_{(\beta_0,0)}$ with norm
less than $3\beta_0^{-1}\Vert\Psi\Vert_\infty$.
For any $u\in\mathbb R^2$, we define $P_u( f)=P(e^{i\langle u,\hat\Psi\rangle} f)$. Observe that
\begin{equation}\label{Puk}
\forall k\ge 1,\ \ \ \ 
P_u^k(f)=P^k(e^{i \langle u,\hat S_k\rangle}f)\ \ \mbox{and}\ \ 
P_u^k(f\circ \hat T^k\times g)=f P_u^k(g),
\end{equation}
with $\hat S_n:=\sum_{k=0}^{n-1}\hat\Psi\circ\hat T^k$.
In \cite{SzV}, Sz\'asz and Varj\'u applied the classical Nagaev-Guivarc'h method \cite{Nag1,Nag2,GH} to this context.
This method plays a crucial role in the proof of Proposition \ref{LEM}
and gives in particular the following inequalities
(see \cite{SzV} and Lemma 12 of \cite{FP09b})
\begin{equation}\label{normebornee}
K_1:=\sup_{u\in[-\pi,\pi^2]}\Vert  P_u^k\Vert <\infty,
\end{equation}
\begin{equation}\label{intnorme}
\exists K>0,\ \ \forall k\ge 1,\ \ \forall h\in\mathcal V,\ \ \ 
(2\pi)^{-2}\int_{[-\pi,\pi]^2}\Vert  P_u^k(h)\Vert\, du   \le \frac{K\Vert  h\Vert }k.
\end{equation}
The following result generalizes Proposition 3 of \cite{FP09b}.
\begin{prop}\label{LEM}
For any $p>1$, 
there exist $C>0$ and $b>0$ such that for any nonnegative
integers $k,n,r,m$,
any $N_1,N_2\in\mathbb Z^2$,
any $A_1,A_2,A_3\subseteq\bar M$ union of components of $\xi_{-k}^k$,
and any
$B\subseteq \bar M$ union of component of $\xi_{-k}^\infty$, we have
$$\left|\cov_{\bar\mu}({\mathbf 1}_{A_1\cap\{S_n=N_1\}\cap \bar T^{-n}A_2},
{\mathbf 1}_{A_3\cap\{S_r=N_2\}\cap\bar T^{-r}B}\circ \bar T^{n+m})\right|
    \le\frac{C\min(1, e^{-am+bk})}{nr}.$$
\end{prop}
\begin{proof}
First, we assume that $2k<\min(n,r)$ and $m>6k$.
Let us write
$$C_{n,m,r}:=\cov_{\bar\mu}({\mathbf 1}_{A_1\cap\{S_n=N_1\}\cap\bar T^{-n}A_2},
{\mathbf 1}_{A_3\cap\{S_r=N_2\}\cap\bar T^{-r}B}\circ\bar T^{n+m}).$$
Observe that $\bar T^{-k}A_i$ is a union of components of $\xi_0^{2k}$
and that $\bar T^{-k}B$ is a union of components of $\xi_0^\infty$.
Let $\hat A_i:=\hat\pi\tilde\pi^{-1}\bar T^{-k}A_i$ 
and $\hat B:=\hat\pi\tilde\pi^{-1}\bar T^{-k}B$.
These sets are measurable and
satisfy $\tilde\pi^{-1}\bar T^{-k}A_i=\hat\pi^{-1}\hat A_i$
and $\tilde\pi^{-1}\bar T^{-k}B=\hat\pi^{-1}\hat B$.
So
\begin{eqnarray*}
C_{n,m,r}&=&{\cov}_{\hat\mu}({\mathbf 1}_{\hat A_1}
  {\mathbf 1}_{\hat S_n=N_1}\circ\hat T^k{\mathbf 1}_{\hat A_2}
 \circ \hat T^n,
({\mathbf 1}_{\hat A_3}{\mathbf 1}_{\hat S_r=N_2}\circ\hat T^k
  {\mathbf 1}_{\hat B}\circ \hat T^{r})\circ \hat T^{n+m})\\
&=&\frac 1{(2\pi)^4}\int_{([-\pi;\pi]^2)^2}
e^{-i\langle u,N_1\rangle}e^{-i\langle t,N_2\rangle}\\
&\ &\times {\cov}_{\hat\mu}({\mathbf 1}_{\hat A_1}
  e^{i\langle u,\hat S_n\rangle}\circ\hat T^k{\mathbf 1}_{\hat A_2}
 \circ \hat T^n,({\mathbf 1}_{\hat A_3}e^{i\langle t,\hat S_r\rangle}
  \circ\hat T^k{\mathbf 1}_{\hat B}\circ
 \hat T^{r})\circ \hat T^{n+m})\, dudt.
\end{eqnarray*}
Now, due to \refeq{Puk},
the covariance appearing in this last integral can
be rewritten
$${\mathbb E}_{\hat\mu}[ P_t^k({\mathbf 1}_{\hat B} P_t^{r-k}
( P^k({\mathbf 1}_{\hat A_3} P^{m-k}(g_u-{\mathbb E}_{\hat\mu}[g_u]))))], $$
with $g_u:= P_u^k({\mathbf 1}_{\hat A_2} P_u^{n-k}
( P^k({\mathbf 1}_{\hat A_1})))$.
Since $\Vert P_t\Vert _{L^1(\bar\mu)}\le 1$,  we obtain
\begin{eqnarray*}
|C_{n,m,r}|&\le& (2\pi)^{-4}\int_{([-\pi,\pi]^2)^2}
{\mathbb E}_{\hat\mu}[|{\mathbf 1}_{\hat B} P_t^{r-k}
( P^k({\mathbf 1}_{\hat A_3} P^{m-k}(g_u-{\mathbb E}_{\hat\mu}[g_u])))|]\, dtdu\\
&\le& (2\pi)^{-4}\int_{([-\pi,\pi]^2)^2}C_0\bar\mu(B)^{\frac 1p}
\Vert  P_t^{r-k}
( P^k({\mathbf 1}_{\hat A_3} P^{m-k}(g_u-{\mathbb E}_{\hat\mu}[g_u])))\Vert \, dtdu\ \ \mbox{by \eqref{normeq}}\\
&\le& (2\pi)^{-2}\int_{[-\pi,\pi]^2}\bar\mu(B)^{\frac 1p}
\frac{KC_0}{r-k} K_1(3\beta_0^{-2k})
\Vert  P^{m-k}(g_u-{\mathbb E}_{\hat\mu}[g_u])))\Vert \, du
\mbox{ by \eqref{intnorme}\eqref{normebornee}\eqref{lem10}\eqref{norm0}} \\
&\le& (2\pi)^{-2}
\bar\mu(B)^{\frac 1p}
\frac{KC_0}{r-k} K_1(3\beta_0^{-2k})
\int_{[-\pi,\pi]^2}K_0e^{-a(m-k)}
\Vert  g_u\Vert \, du \ \mbox{by \refeq{expo}}\\
&\le& (2\pi)^{-2}\bar\mu(B)^{\frac 1p}
\frac{KC_0}{r-k} K_1(3\beta_0^{-2k})K_0e^{-a(m-k)}
\int_{[-\pi,\pi]^2}
\Vert   P_u^k({\mathbf 1}_{\hat A_2} P_u^{n-k}
( P^k({\mathbf 1}_{\hat A_1})))  \Vert \, du\\
&\le& \bar\mu(B)^{\frac 1p}
\frac{K^2C_0}{(r-k)(n-k)} (3K_1\beta_0^{-2k})^3K_0e^{-a(m-k)}
\ \mbox{by \refeq{normebornee}\eqref{intnorme}\eqref{lem10}\eqref{norm0}}\\
&\le& \frac{\hat C_0 e^{-am+b_0k}}{nr},
\end{eqnarray*}
for some $b_0>0$.
We still assume that $m>6 k$.
When $n\le 2k$ and $r>2 k$, we observe that $A_1\cap\{S_n=N_1\}\cap \bar T^{-n}A_2$
is a union of components of $\xi_{-k}^{3k}$, using the same argument
we obtain an upper bound in $\hat C_0 e^{-am+3b_0k}/r$
which is less than $\hat C_1 e^{-am+4b_0k}/(nr)$ for some $\hat C_1>0$. Treating analogously the cases $(r\le 2k;\ 2k < n)$ and  $(n\le 2k;\ r \le 2k)$, we obtain the following bound
\begin{equation}\label{interm}
|C_{n,m,r}|\le\frac{\hat C e^{-am+bk}}{nr},
\ \ \mbox{for some }\hat C>0\mbox{ and some } b\ge 6a>0.
\end{equation}
Assume now that $am\le bk$ (this is true if $m\le 6k$). Then,
due to the fact that $|\cov_{\bar\mu}(f,g)|\le 
|{\mathbb E}_{\bar\mu}[fg]|+|{\mathbb E}_{\bar\mu}[f]{\mathbb E}
_{\bar\mu}[g]|$, we have
\begin{eqnarray*}
|C_{n,m,r}|
&\le&\bar\mu(S_n=N_1;S_r\circ \bar T^{n+m}=N_2)+\bar\mu(S_n=N_1)\bar\mu(S_r=N_2)\\
&\le&|\cov_{\bar\mu}(\mathbf 1_{S_n=N_1},{\mathbf 1}_{S_r=N_2}\circ\bar T^{n+m})|+2\bar\mu(S_n=N_1)\bar\mu(S_r=N_2)
\le\frac{\hat C_2}{nr},
\end{eqnarray*}
using estimation (\ref{interm}) with $k=0$ and the local limit theorem for $S_n$ (see \cite{SzV} or  \refeq{eqprop4}).
\end{proof}
%
%
%
%
%
\section{Estimate of the variance of $V_n$}\label{variance}
Recall that $\Sigma^2$ is invertible. 
In particular, there exists $\tilde a_0$ such that
$\langle (\Sigma^2)^{-1}x,x\rangle\ge 2\tilde a_0|x|^2$
for every $x\in\mathbb R^2$.
Comparing
$$\sum_{x\in\mathbb Z^2\ :\ |x|\le am}e^{-\frac{\langle(\Sigma^2)^{-1} x,x\rangle}{2m}}\ \ \mbox{with}\ \ \int_{|u|\le am} 
e^{-\frac{\langle(\Sigma^2)^{-1} u,u\rangle}{2m}}\, du,$$
we obtain the 
following useful formula
\begin{equation}\label{clef}
\sup_{||S_1||_\infty \le a\le 3||S_1||_\infty}
  \left|\sum_{x\in\mathbb Z^2\ :\ |x|\le am}e^{-\frac{\langle(\Sigma^2)^{-1} x,x\rangle}{2m}} -
 2\pi m\sqrt{\det\Sigma^2}\right|=O(\sqrt{m}).
\end{equation}
\begin{proof}[Proof of Proposition \ref{prop1}]
As in \cite{Bol}, the proof of Proposition \ref{prop1} is based on the following formula
$$
\var (V_n)=4\sum_{1\le k_1<\ell_1 \le n}\sum_{1\le k_2<\ell_2 \le n}D_{k_1,\ell_1,k_2,\ell_2}
= 8A_1+8A_2+8A_3+4A_4,
$$
with 
$D_{k_1,\ell_1,k_2,\ell_2}:=\bar\mu(E_{k_1,\ell_1}\cap E_{k_2,\ell_2})
       -\bar\mu(E_{k_1,\ell_1})\bar\mu( E_{k_2,\ell_2})$ and 
$$A_1:=\sum_{1\le k_1<\ell_1\le k_2<\ell_2\le n}D_{k_1,\ell_1,k_2,\ell_2},\ 
A_2:=\sum_{1\le k_1\le k_2<\ell_1\le\ell_2\le n}D_{k_1,\ell_1,k_2,\ell_2},$$ 
$$A_3:=\sum_{1\le k_1< k_2<\ell_2<\ell_1\le n}D_{k_1,\ell_1,k_2,\ell_2}, \ 
A_4:=\sum_{1\le k_1<\ell\le n}[\bar\mu(E_{k_1,\ell})
       -(\bar\mu(E_{k_1,\ell}))^2].$$
We use the notations and ideas of the proof of Proposition \ref{prop2}. 
Let $p\in(1,2)$. We take $m,k$ such that $m^2=\delta^{-k}=n^{1/100}$.
We have
$$
\bar\mu(E_{k_1,\ell_1}\cap E_{k_2,\ell_2})=\sum_{\mathcal C,\mathcal C'\in\mathcal P_m}
\bar\mu(\mathcal C\cap E_{0,\ell_1-k_1}\cap \bar T^{-(k_2-k_1)}(
\mathcal C'\cap E_{0,\ell_2-k_2})).$$
As in the proof of Proposition \ref{prop2}, we approximate
$\mathcal C\cap E_{0,r}$ by $\tilde E_{r,\mathcal C}$.
See \refeq{EnC} and \refeq{DnC} for the definition
of $\tilde E_{r,\mathcal C}$ and of $\tilde{\mathcal D}_{r,\mathcal C}$.
We recall that
$(\mathcal C\cap E_{0,r})\triangle \tilde E_{r,\mathcal C}
\subseteq  \tilde{\mathcal D}_{r,\mathcal C}$ and that, according to
\refeq{eqprop4bis}, if
$r\ge 3k$, we have (for $p>1$ large enough)
\begin{equation}\label{mesE}
\bar\mu(\tilde E_{r,\mathcal C})=O\left(\frac{m^{-2}}r+
     \frac{k m^{-2/p}}{r^{\frac 32}}\right)=O(m^{-2}r^{-1})=O(r^{-1}n^{-\frac{1}{100}})
\end{equation}
and
\begin{equation}\label{mesD}
\bar\mu(\tilde{\mathcal D}_{r,\mathcal C})\le\frac{m^{-2} 
\delta^k}{r}+\frac{k m^{-2/p}}{r^{\frac 32}}=O(m^{-2}r^{-1}\delta^k)=O(r^{-1}n^{-\frac{2}{100}}).
\end{equation}
\begin{itemize}
\item \underline{Control of $A_1$}. 
We have

$|\cov_{\bar\mu}({\mathbf 1}_{\mathcal C\cap E_{0,r}},
      {\mathbf 1}_{\mathcal C'
       \cap E_{0,s}}\circ  \bar T^{r+\ell})-
    \cov_{\bar\mu}({\mathbf 1}_{\tilde E_{r,\mathcal C}},
        {\mathbf 1}_{\tilde E_{s,\mathcal C'}}\circ \bar T^{r+\ell})|
   \le$
\begin{multline*}
\le |\cov_{\bar\mu}({\mathbf 1}_{\tilde E_{r,\mathcal C}\cup
      \tilde{\mathcal D}_{r,\mathcal C}},
 {\mathbf 1}_{\tilde{\mathcal D}_{s,\mathcal C'}}\circ\bar  T^{r+\ell}))|+
  |\cov_{\bar\mu}({\mathbf 1}_{\tilde{\mathcal D}_{r,\mathcal C}},
   {\mathbf 1}_{\tilde E_{s,\mathcal C'}\cup
      \tilde{\mathcal D}_{s,\mathcal C'}}\circ \bar T^{r+\ell})|\\
   +2{\bar\mu}({\tilde E_{r,\mathcal C}\cup
      \tilde{\mathcal D}_{r,\mathcal C}})\bar\mu
  ({\tilde{\mathcal D}_{s,\mathcal C'}})+
  2{\bar\mu}({\tilde{\mathcal D}_{r,\mathcal C}})\bar\mu
   ({\tilde E_{s,\mathcal C'}\cup
      \tilde{\mathcal D}_{s,\mathcal C'}}).
\end{multline*}
Now, due to \refeq{tildeE}, \refeq{tildeD}, 
applying Proposition \ref{LEM} (together with \refeq{mesE}
and \refeq{mesD}), we obtain
$$
\sum_{\mathcal C}\sum_{\mathcal C'}
|\cov_{\bar\mu}({\mathbf 1}_{\mathcal C\cap E_{0,r}},
      {\mathbf 1}_{\mathcal C'
       \cap E_{0,s}}\circ\bar T^{r+\ell})|
\le m^4
     \frac{C\min(1, e^{-a\ell+bk})}{rs}+\frac{Cn^{-\frac 1{100}}}
    {rs},
$$
and so (considering separately the sums over $\ell$
such that $a\ell\ge 2bk$ and $a\ell<2bk$ )
\begin{equation}\label{A1}
A_1=\sum_{k_1\ge 1,r>0,\ell\ge 1,s>0:k_1+r+\ell+s\le n}
\cov_{\bar\mu}({\mathbf 1}_{ E_{0,r}},
      {\mathbf 1}_{
        E_{0,s}}\circ\bar T^{r+\ell})=O(n^{2-\frac 1{100}}\log^2 n).
\end{equation}
\item \underline{Control of $A_2$}.
Notice that
$$A_2=\sum_{k_1+r+\ell+s\le n}
\cov_{\bar\mu}(E_{0,r+\ell},E_{r,r+\ell+s}) $$
(where the sum is also taken over $k_1\ge 1$, $r\ge 0$, $\ell\ge 1$,
$s\ge 0$).
According to Proposition \ref{prop2}, we have
$\bar\mu(E_{0,r})= \frac c {2r}+O(r^{-1-\eta})$ with $\eta>0$.
A direct computation (see Lemma \ref{A20}) gives
$\sum_{k_1+r+\ell+s\le n}\frac 1{(r+\ell)(\ell+s)}\sim\frac{\pi^2}{12}n^2$. 
Hence
\begin{equation}
\sum_{k_1+r+\ell+s\le n}\bar\mu(E_{0,r+\ell})
\bar\mu(E_{0,\ell+s})\sim\frac{\pi^2}{12}\frac{c^2}4n^2.
\end{equation}
Now, let us prove that
\begin{equation}\label{A2terme1}
\sum_{k_1+r+\ell+s\le n}\bar\mu(E_{0,r+\ell}\cap
 E_{r,r+\ell+s})\sim J\frac{c^2}4n^2.
\end{equation}
From which we conclude that
\begin{equation}\label{A2}
A_2\sim\left(J-\frac{\pi^2}{12}\right)\frac{c^2}4n^2.
\end{equation}
We have to estimate
$C^{(2)}_{r,\ell,s}:= \bar\mu(E_{0,r+\ell}\cap E_{r,r+\ell+s})$.
Given $\mathcal C,\mathcal C'\in \mathcal P_m$, we
consider the set
$\mathcal E_{r,\ell,s,\mathcal C,\mathcal C'}^{(2)}:=
 E_{0,r+\ell}\cap E_{r,r+\ell+s}\cap\mathcal C\cap
 \bar T^{-r}\mathcal C'$
which we approximate by
$\tilde{\mathcal E}_{r,\ell,s,\mathcal C,\mathcal C'}^{(2)}:=
 \tilde E_{r+\ell,\mathcal C}\cap  \bar T^{-r} \tilde E_{\ell+s,\mathcal C'}$.
We notice that
\begin{equation}\label{diffsym2}
\mathcal E_{r,\ell,s,\mathcal C,\mathcal C'}^{(2)}\triangle\tilde{\mathcal E}_{r,\ell,s,\mathcal C,\mathcal C'}^{(2)}\subseteq 
(\tilde  {\mathcal D}_{r+\ell,\mathcal C}\cap\bar T^{-r} (
\tilde {\mathcal D}_{\ell+s,\mathcal C'}\cup \tilde E_{\ell+s,\mathcal C'}))\cup ((\tilde  {E}_{r+\ell,\mathcal C}
\cup \tilde  {\mathcal D}_{r+\ell,\mathcal C})\cap\bar T^{-r} \tilde {\mathcal D}_{\ell+s,\mathcal C'}).
\end{equation}
Observe that $\bar\mu(\tilde{\mathcal E}_{r,\ell,s,\mathcal C,
\mathcal C'}^{(2)})$
is equal to the following sum
$\sum_x G_{r,\ell,s,\mathcal C,\mathcal C',x}$
(where $\sum_x$ means the sum over the $x\in\mathbb Z^2$
such that $|x|\le\min(r,\ell+1,s+2)\Vert S_1\Vert_\infty$)
with 
\begin{multline}\label{sumA2}
G_{r,\ell,s,\mathcal C,\mathcal C',x}:=
\sum_{|N|,|N|'\le \Vert S_1\Vert_\infty}\bar\mu 
(\tilde{\mathcal C}\cap\{S_r=x\}\cap\bar T^{-r}(
   \tilde{\mathcal C}'\cap\{S_\ell=N-x\}\cap\\
  \cap\bar T^{-\ell}
    (\bar M\cap(\tilde V-N)\cap\{S_s=x+N'-N\}\cap\bar T^{-s}
(\bar M\cap(\tilde V'
     -N')))),
\end{multline}
and where $\tilde{\mathcal C}$, $\tilde{\mathcal C}'$,
$\tilde V$ and $\tilde V'$ are the $\xi_{-k}^k$-measurable sets
such that
$\tilde E_{r+\ell,\mathcal C}=\tilde{\mathcal C}\cap T^{-r-\ell}\tilde{ V}$ and $\tilde E_{\ell+s,\mathcal C}=\tilde{\mathcal C}'\cap T^{-\ell-s}\tilde{V}'$ (see \refeq{EnC}).
Due to (\ref{diffsym2}), we have
\begin{equation}\label{sumA2err}
\bar\mu(\mathcal E_{r,\ell,s,\mathcal C,\mathcal C'}^{(2)}\triangle\tilde{\mathcal E}_{r,\ell,s,\mathcal C,\mathcal C'}^{(2)})\le \mathcal \sum_x (S_{1,x}+\mathcal S_{2,x}),
\end{equation}
where $\mathcal S_{1,x}$ (resp. $\mathcal S_{2,x}$) is obtained from
(\ref{sumA2}) by replacing $\tilde{ V}$ and 
$\tilde{ V}'$ by $\tilde{\mathcal D}_{\mathcal C}$ and
$\tilde{ V}'\cup \tilde{\mathcal D}_{\mathcal C'}$
(resp. by $\tilde{ V}\cup \tilde{\mathcal D}_{\mathcal C}$
and $\tilde{\mathcal D}_{\mathcal C'}$), with the notation
$\tilde{\mathcal D}_{\mathcal C}$ introduced in \refeq{DnC}. 
To estimate $\bar\mu(\tilde{\mathcal E}_{r,\ell,s,\mathcal C,
\mathcal C'}^{(2)})$
and
$\bar\mu(\mathcal E_{r,\ell,s,\mathcal F}^{(2)}\triangle\tilde{\mathcal E}_{r,\ell,s,\mathcal C,\mathcal C'}^{(2)})$, we will apply \refeq{eqprop4} three successive times to each summand appearing in (\ref{sumA2})
and in (\ref{sumA2err}).

We start with the study of (\ref{sumA2}). According
to \refeq{eqprop4} and since $|N|,|N'|\le \Vert S_1\Vert_\infty$, when $r,\ell,s\ge 3k$, the quantity given by
(\ref{sumA2}) is equal to
\begin{equation}\label{A11}
\frac{\bar\mu(\tilde{\mathcal C})\bar\mu(\odot)}{\sqrt{\det\Sigma^2}
2\pi r}e^{-\frac{\langle (\Sigma^2)^{-1}x,x\rangle}{2r}}+e_1,
\end{equation}
with
\begin{equation}\label{A12}
\bar\mu(\odot)=\frac{\bar\mu(\tilde{\mathcal C}')\bar\mu(\otimes)}{\sqrt{\det\Sigma^2}
2\pi \ell}e^{-\frac{\langle (\Sigma^2)^{-1}x,x\rangle}{2\ell}}+e_2,
\end{equation}
\begin{equation}\label{A13}
\bar\mu(\otimes)=\frac{\bar\mu(\tilde {V})\bar\mu(\tilde{ V}')}{\sqrt{\det\Sigma^2}
2\pi s}e^{-\frac{\langle (\Sigma^2)^{-1}x,x\rangle}{2s}}+e_3,
\end{equation}
the error terms being estimated by
$$|e_1|\le\tilde K_1k \frac{\bar\mu(\odot)^{\frac 1p}}{r^{\frac32}},\ \ \ \ 
|e_2|\le\tilde K_1k \frac{\bar\mu(\otimes)^{\frac 1p}}{\ell^{\frac 32}}\ \mbox{and}\ \ \ \ |e_3|\le\tilde K_1ks^{-\frac 32} ,$$
for some $\tilde K_1>1$.
So the contribution to $A_2$ of the three dominating terms in \refeq{A11}, \refeq{A12} and \refeq{A13} is
(where $\sum^+$ means the sum restricted to $k_1\ge 1$,
$\min(r,s,\ell)\ge 3k$):
$$\sum^+_{k_1+r+\ell+s\le n} \sum_x\sum_{\mathcal C,\mathcal C'}
\frac{\bar\mu(\tilde{\mathcal C})\bar\mu(\tilde{\mathcal C}')
\bar\mu(\tilde{ V})\bar\mu(\tilde{ V}')}{{(\det\Sigma^2)}^{\frac 32}
(2\pi)^3 r\ell s}e^{-\frac 12{\langle (\Sigma^2)^{-1}x,x\rangle}
(\frac 1r+\frac 1\ell+\frac 1s)}.$$
$$
= \sum_{\mathcal C,\mathcal C'}
\frac{4\mathbb E_{\bar\mu}[\tau{\mathbf 1}_{\mathcal C}]
\mathbb E_{\bar\mu}[\tau{\mathbf 1}_{\mathcal C'}](1+o(1))}
{(\det\Sigma^2)^{\frac 32}
(2\pi)^{3}(\sum_i|\partial O_i|)^2}
 \sum^+_{k_1+r+\ell+s\le n}\sum_x\frac{e^{-\frac 12{\langle (\Sigma^2)^{-1}x,x\rangle}
(\frac 1r+\frac 1\ell+\frac 1s)}}{ r\ell s}.
$$
Since
$1/\min(r,\ell,s)\le\frac 1r+\frac 1\ell+\frac 1s\le
   3/\min(r,\ell,s)$, due to \refeq{clef}, we have
\begin{eqnarray*}
\sum^+_{k_1+r+\ell+s\le n}\sum_x\frac{e^{-\frac 12{\langle (\Sigma^2)^{-1}x,x\rangle}
(\frac 1r+\frac 1\ell+\frac 1s)}}{ 
  2\pi\sqrt{\det\Sigma^2}r\ell s}
&=&
\sum^+_{k_1+r+\ell+s\le n}\left(
\frac{1}{r\ell+rs+s\ell} +\frac{O\left(\sqrt{\min(r,\ell,s)}\right)}{r\ell s}\right)\\
&=& O(n^{\frac 32})+\sum^+_{k_1+r+\ell+s\le n}\frac{1}{r\ell+rs+s\ell}\\
&=&o(n^2)+\sum_{k_1+r+\ell+s\le n}\frac{1}{r\ell+rs+s\ell},
\end{eqnarray*}
(where the last sum is taken over $k_1,r,\ell,s\ge 1$) since
$$ \sum_{k_1,r,s=1}^n
  \sum_{\ell= 1}^{3k}\frac{1}{r\ell+rs+s\ell}
    \le O(n\log n)\sum_{r,s=1}^n\frac 1{rs}=O(n\log^3n)=o(n^2).$$
Now, according to Lemma \ref{controlA2}, we have
$$ \sum_{k_1+r+\ell+s\le n}\frac 1{r\ell+r s+s\ell}\sim  n^2 J.$$
We finally obtain that the contribution to $A_2$ 
of (\ref{sumA2}) coming from the dominating terms of
\refeq{A11}, \refeq{A12} and \refeq{A13} is
\begin{equation}
\sim  J\frac{c^2}4n^2.
\end{equation}
Now, we prove that the other contributions are in $o(n^2)$.
\begin{itemize}
\item Using the fact that $|x|\le 2\min(r,\ell,s)\Vert S_1\Vert_\infty$, 
we get that the contribution to $A_2$ of 
the term coming from the composition of the three error terms $(e_1,e_2,e_3)$ is bounded by
$$m^4\sum_{k_1+r+\ell+s\le n}
 \sum_x\frac{\tilde K_1^3k^3}{r^{\frac 32}\ell^{\frac 3{2p}}s^{\frac 3
    {2p^2}}}\le 16(\Vert S_1\Vert_\infty+1)^4
  \tilde K_1^3n^{\frac 2{100}}n k^3 \sum_{r+\ell+s\le n} 
    \frac{\min(r^2,\ell^2,s^2)}{r^{\frac 32}\ell^{\frac 3{2p}}s^{\frac 3
    {2p^2}}}.$$
$$ \le 
    O\left(n^{\frac {102}{100}}\log^3n \sum_{r,\ell,s\le n} 
    \frac{r^{\frac 23}\ell^{\frac 23}s^{\frac 23}}
   {r^{\frac 32}\ell^{\frac 3{2p}}s^{\frac 3
    {2p^2}}}\right)=O(n^{\frac{102}{100}+5-\frac{9}{2p^2}})=o(n^2),$$
if we take $p>1$ small enough.
\item Analogously, the contribution to $A_2$ of the composition of one dominating term and of two error terms of \refeq{A11}, \refeq{A12} and
\refeq{A13} is less (up to a multiplicative constant) than 
\begin{eqnarray*}
n^{\frac 2{100}}k^2
\sum_{k_1+r+\ell+s\le n}\frac{1}{r^{\frac 1{p^2}}\ell^{\frac 32} s^{\frac 3{2p}}}\sum_xe^{-\tilde a_0^2\frac{|x|^2}{rp^2}} &\le&
O\left(n
^{\frac{102}{100}}k^2\sum_{r,\ell,s\le n}\frac{\min(r,\ell^2,s^2)}
{r^{\frac 1{p^2}}\ell^{\frac 32} s^{\frac 3{2p}}}\right)\\
&\le& O\left(n
^{\frac{102}{100}}k^2\sum_{r,\ell,s\le n}\frac{r^{\frac 12}(\ell^2)^{\frac 14}(s^2)^{\frac 14}}
{r^{\frac 1{p^2}}\ell^{\frac 32} s^{\frac 3{2p}}}\right)\\
&=&O(n^{\frac{102}{100}+\frac 32-\frac 1{p^2}+\frac 32-\frac 3{2p}}\log^3n)=o(n^2),
\end{eqnarray*}
if we take $p>1$ small enough.
\item Now, the contribution to $A_2$ of the composition of two dominating terms and of one error term of \refeq{A11}, \refeq{A13} and
\refeq{A13} is less (up to a multiplicative constant) than 
$$n^{\frac{102}{100}}\sum_{r,\ell,s\le n}\frac{k^2}{r^{\frac 1{p}}\ell^{\frac 1{p}} s^{\frac 3{2}}}\sum_xe^{-\frac{\tilde a_0^2}p
{|x|^2}\frac{r+\ell}{r\ell}}.$$
On the one hand, we have
$ \sum_xe^{-a_0^2{|x|^2}\frac{r+\ell}{r\ell}}\le \min(r,\ell)$
(using the fact that $r\ell/(r+\ell)\le\min(r,\ell)$).
On the other hand, this sum is in $O(s^2)$. Therefore
the quantity we are looking at is less than
$$
O\left(n^{\frac{102}{100}}k^2
\sum_{r,\ell,s\le n}\frac{\min(r,\ell,s^2)}{r^{\frac 1{p}}\ell^{\frac 1{p}} s^{\frac 3{2}}}\right)=
O\left(n^{\frac{102}{100}}k^2
\sum_{r,\ell,s\le n}\frac{r^{\frac 38}\ell^{\frac 38}
(s^2)^{\frac 14}}
 {r^{\frac 1{p}}\ell^{\frac 1{p}} s^{\frac 3{2}}}\right)
$$
$$=O\left(n^{\frac{102}{100}+2+\frac 3{4}-\frac 2p}\log^3n\right)=o(n^2), $$
if $p>1$ is small enough.
\item 
If $r\le 4k$ or $\ell\le 4k$ or $s\le 4k$, then 
$\sum_x$ is a sum over $|x|\le 4k\Vert S_1\Vert_\infty$ and
one of the following sets is $\xi_{-5k}^{5k}$-measurable:
$$
 \tilde{\mathcal C}\cap\{S_r=x\}\cap\bar T^{-r}
   \tilde{\mathcal C}'\ \ \mbox{or}\ \ 
  \tilde{\mathcal C}'\cap\{S_\ell=N-x\}\cap\bar T^{-\ell}(\bar M\cap(\tilde V-N))$$ 
$$     \mbox{or}\ \ 
       \bar M\cap(\tilde V-N)\cap\{S_s=x+N'-N\}\cap\bar T^{-s}
    (\bar M\cap(\tilde V'-N')).$$
We then apply \refeq{eqprop4} accordingly and 
take in account the fact that the sum on $r$ or $k$ or $\ell$
must be taken on $\{1,...,4k\}$. This leads to a term in $o(n^2)$.
\item 
Finally, the estimate of (\ref{sumA2err}) follows the same lines as the
estimate of (\ref{sumA2}). We obtain an analogous estimation multiplied by $\delta^{k}$. This ensures that the contribution of (\ref{sumA2err}) to $A_2$ is in $o(n^2)$.
\end{itemize}
\item \underline{Control of $A_3$}.
We have
$$A_3=\sum_{k_1+r+\ell+s\le n}
\cov_{\bar\mu}({\mathbf 1_{E_{0,r+\ell+s}}},{\mathbf 1_{E_{0,\ell}}}\circ \bar T^r).$$
This part is the most delicate. 
Indeed the terms
$$\sum_{k_1+r+\ell+s\le n}\bar\mu(E_{0,r+\ell+s})
\bar\mu(E_{0,\ell})\ \ \ \ \mbox{and}\ \ 
\sum_{k_1+r+\ell+s\le n}\bar\mu(E_{0,r+\ell+s}\cap E_{r,r+\ell})$$
are in $n^2\log n$. But we will prove that their difference
is in $n^2$. More precisely, we show that
\begin{equation}\label{A3}
A_3\sim\frac{c^2}8 n^2.
\end{equation}
First, according to Proposition \ref{prop2}, we have
\begin{eqnarray}
\sum_{k_1+r+\ell+s\le n}\bar\mu(E_{0,r+\ell+s})
\bar\mu(E_{0,\ell}) &=&\sum_{k_1+r+\ell+s\le n}\left(\frac {c+
      O((r+\ell+s)^{-\eta})}{2(r+\ell+s)}\right)
\bar\mu(E_{0,\ell})\nonumber\\
&=&o(n^2)+\sum_{\mathcal C}\sum_{k_1+r+\ell+s\le n}\frac {c\bar\mu(E_{0,\ell}\cap\mathcal C)}{2(r+\ell+s)}\nonumber\\
&=&o(n^2)+\sum_{\mathcal C}\sum_{k_1+r+\ell+s\le n}\frac {c\bar\mu(\tilde E_{\ell,\mathcal C})+O(n^{-\frac{2}{100}}/\ell)}{2(r+\ell+s)}\nonumber\\
&=&o(n^2)+\sum_{\mathcal C}\sum_{k_1+r+\ell+s\le n}\frac {c\bar\mu(\tilde E_{\ell,\mathcal C})}{2(r+\ell+s)}.\label{A30}
\end{eqnarray}
Indeed, setting $q=\ell+r$ and $t=\ell+r+s$, we have
$$\sum_{\mathcal C}\sum_{k_1+r+\ell+s\le n}\frac {n^{-\frac{2}{100}}}{\ell(r+\ell+s)}\le
  n^{1-\frac 1{100}}\sum_{t=1}^n\frac 1t\sum_{q=1}^t\sum_{\ell=1}^q
   \frac 1{\ell}=O(n^{2-\frac 1{100}}\log n)
  =o(n^2).$$
Now, let us estimate
$\sum_{k_1+r+\ell+s\le n}\bar\mu(E_{0,r+\ell+s}
\cap E_{r,r+\ell})$ in terms of $\bar\mu(\tilde E_{\ell,\mathcal C})$.
For any $\mathcal C,\mathcal C'\in\mathcal P_m$, 
we approximate once again $\mathcal C'\cap E_{0,r+\ell+s}\cap \bar T^{-r}\mathcal C
\cap E_{r,r+\ell}$ by
\begin{equation}\label{approx}
 \tilde E_{r+\ell+s,\mathcal C'}\cap \bar T^{-r}\tilde E_{\ell,\mathcal C},
\end{equation}
the measure of which is
$\sum_xH_{r,\ell,s,\mathcal C,\mathcal C',x}$
(with $\sum_x$ being taken on the set of $x\in\mathbb Z^2$ such that
$|x|\le \min(r,s+2)\Vert S_1\Vert_\infty$) and with
\begin{multline}\label{HHH}
H_{r,\ell,s,\mathcal C,\mathcal C',x}:=
\sum_{|N|,|N'|\le L}\bar\mu(\tilde{\mathcal C}'\cap\{S_r=x\}
\cap\bar T^{-r}(\tilde{\mathcal C}\cap\{S_\ell=N\}\cap\\
    \cap\bar T
   ^{-\ell}(\bar M\cap(\tilde V-N)\cap\{S_s=N'-x-N\}\cap \bar T^{-s}(\bar M\cap(\tilde V'-N'))))) .
\end{multline}
Now, applying \refeq{eqprop4} and \refeq{eqprop4bis} (when $\min(r,s)\ge 3k$), we obtain
that this quantity is equal to
\begin{equation}\label{A31}
\frac{\bar\mu(\tilde{\mathcal C}')\bar\mu(\Diamond)}{\sqrt{\det\Sigma^2}
2\pi r}e^{-\frac{\langle (\Sigma^2)^{-1}x,x\rangle}{2r}}+e'_1,
\end{equation}
with
\begin{equation}\label{A32}
\bar\mu(\Diamond)=\frac{\bar\mu(\tilde E_{\ell,\mathcal C})\bar\mu(\tilde V')}{\sqrt{\det\Sigma^2}
2\pi s}e^{-\frac{\langle (\Sigma^2)^{-1}x,x\rangle}{2s}}+e'_2,
\end{equation}
the error terms being estimated by
$$|e'_1|\le\tilde K_1k \frac{\bar\mu(\Diamond)^{\frac 1p}}{r^{\frac32}}\ \ \  \mbox{and}\ \ 
|e'_2|\le\tilde K_1k \frac{\bar\mu(\tilde E_{\ell,\mathcal C})^{\frac 1p}}{s^{\frac 32}}.$$
We obtain that the contribution to $A_3$ of the dominating terms of
\refeq{A30}, \refeq{A31} and \refeq{A32} is
(where $\sum^*$ stands for the sum over $k_1\ge 1$, $\ell\ge 1$
and $\min(r,s)\ge 3k$)
$$
\sum_{\mathcal C,\mathcal C'} \sum^*_{k_1+r+\ell+s\le n}
\left(\frac{\bar\mu(\tilde{\mathcal C}')\bar\mu(\tilde V')\bar\mu
    (\tilde E_{\ell,\mathcal C})
   \sum_x e^{-\frac{\langle(\Sigma^2)^{-1}x,x\rangle}{2}\frac{r+s}{rs}}}{
   \det\Sigma^2(2\pi)^2rs}-\frac{c\bar\mu(\tilde{\mathcal C}')\bar\mu(\tilde E_{\ell,\mathcal C})}{2(r+\ell+s)}\right)$$
\begin{eqnarray*}
&=& \sum_{k_1+r+\ell+s\le n}^*\sum_{\mathcal C}\bar\mu(\tilde E_{\ell,
        \mathcal C})\frac c{2}
\left(\frac{(1+O(n^{-\frac 1{200}}))
   \sum_x
  e^{-\frac{\langle(\Sigma^2)^{-1}x,x\rangle}{2}\frac{r+s}{rs}}}{
   \sqrt{\det\Sigma^2} 2\pi rs}-\frac{1}{r+\ell+s}\right)\\
&=&o(n^2)+\sum_{k_1+r+\ell+s\le n}^*\frac{c^2}{4\ell}
\left(\frac{\sum_xe^{-\frac{\langle(\Sigma^2)^{-1}x,x\rangle}{2}\frac{r+s}{rs}}}{
   \sqrt{\det\Sigma^2} 2\pi rs}-\frac{1}{r+\ell+s}\right)\\
&=&o(n^2)+\sum_{k_1+r+\ell+s\le n}^*\frac{c^2}{4\ell}
\left(\frac{1}{ r+s}-\frac{1}{r+\ell+s}\right)\ \ \mbox{due to 
\refeq{clef}}\\
&=&o(n^2)+\frac{c^2}{4}\sum_{k_1+r+\ell+s\le n}^*
\left(\frac{1}{ (r+s)(r+\ell+s)}\right)\\
&=&o(n^2)+\frac{c^2}{4}\sum_{k_1,r,\ell,s\ge 1\ :\ k_1+r+\ell+s\le n}
\left(\frac{1}{ (r+s)(r+\ell+s)}\right)\\
&\sim&\frac{c^2}4n^2\int_{[0,1]^4}\frac{\mathbf{1}_{\{t+u+v+w<1\}}\, dt\, du\, dv\, dw}{(u+w)(u+v+w)}
=\frac{c^2}{8}n^2.
\end{eqnarray*}
For the third line, we used the fact that $\sum_{\mathcal C}\bar\mu
(\tilde E_{\ell,\mathcal C})=\frac c{2\ell}+O(\ell^{-1-\eta})$.
For the last line, we used the Lebesgue dominated convergence theorem and the following equalities obtained by a change of variable
($r=u+w$, $s=u+v+w$) 
and by integrating in $t$, $u$, $r$ and finally in $s$:
\begin{eqnarray*}
\int_{[0,1]^4}\frac{\mathbf{1}_{\{t+u+v+w<1\}}\, dt\, du\, dv\, dw}{(u+w)(u+v+w)}
    &=&\int_{[0,1]^4}\frac{\mathbf{1}_{\{u<r<s,t+s<1\}}\, dt\, du\, dr\, ds}{rs}\\
&=&\int_{0\le u\le r\le s\le 1}\frac{(1-s)\, du\, dr\, ds}{rs}\\
&=&\int_0^1(1-s)\, ds=\frac 12. 
\end{eqnarray*}
Now, it remains to show that the contribution to $A_3$ of all the other
terms is in $o(n^2)$.
\begin{itemize}
\item According to \refeq{mesE}, \refeq{A31} and \refeq{A32}, the contribution of the term coming from the composition of the
two
error terms $e'_1$ and $e'_2$ is in
\begin{eqnarray}
\sum_{\mathcal C,\mathcal C'}\sum_{k_1+r+\ell+s\le n}\sum_x
   k^2\frac{\bar\mu(\tilde E_{\ell,\mathcal C})^{\frac 1{p^2}}}{
       r^{\frac 32}s^{\frac 3{2p}}}
&=&4\sum_{\mathcal C,\mathcal C'}\sum_{k_1+r+\ell+s\le n}
\frac{k^2 n^{-\frac 1{100p^2}}}
  {\ell^{\frac 1{p^2}}r^{\frac 32}s^{\frac 3{2p}}}\min(r^2,s^2)\label{somme}\\
&=&4\sum_{\mathcal C,\mathcal C'}\sum_{k_1+r+\ell+s\le n}
\frac{k^2 n^{-\frac 1{100p^2}}}
  {\ell^{\frac 1{p^2}}r^{\frac 32}s^{\frac 3{2p}}}rs\nonumber\\
&=&O(n^{1+\frac 1{100}(2-\frac 1{p^2})+\frac 12+1-\frac 1{p^2}
   +2-\frac 3{2p}}\log^2n),\nonumber
\end{eqnarray}
which is not enough to conclude.
Hence, we use the estimate of $e'_2$ given by 
Remark \ref{remprop4} for $x\ge 3k$. On the one hand, the last term in the RHS of the formula given
in Remark \ref{remprop4} brings \refeq{somme}
with $s^{\frac 3{2p}}$ replaced by $ks^{\frac 2p}$, which gives $o(n^2)$ for $p>1$ small enough.
On the other hand, the first term in the RHS of the formula of 
Proposition \ref{prop4} gives still $s^{\frac 3{2p}}$, but with
$\min(r^2,s)\le s^{\frac 12}r$ instead of $\min(r^2,s^2)\le rs$.
This ensures that this term is in $o(n^2)$.
\item The contribution of the term coming from the composition of the
error term $e'_1$ of \refeq{A31} and of the dominating
term of \refeq{A32} is in

$\sum_{\mathcal C,\mathcal C'}\sum_{k_1+r+\ell+s\le n}\sum_x\frac{k}{r^{\frac 32}}
  \left(\frac{\bar\mu(\tilde E_{\ell,\mathcal C})\bar\mu(\tilde V')
     e^{-\tilde a_0\frac{|x|^2}s}}{s}\right)^{\frac 1p}$
\begin{eqnarray*}
&=&O\left(n^{\frac 2{100}}\log n\sum_{k_1+r+\ell+s\le n}\frac {\min(s,r^2)}
       {r^{\frac 32}\ell^{\frac 1 p}s^{\frac 1p}}
   \right)\\
&=&O\left(n^{1+\frac 2{100}}\log n\sum_{r+\ell+s\le n}\frac {\sqrt{r}s^{\frac 34}}
       {r^{\frac 32}\ell^{\frac 1 p}s^{\frac 1p}}\right)\\
&=&O\left((\log n)^2n^{1+\frac 2{100}+1-\frac 1p+\frac 74-\frac 1p}\right)=o(n^2),
\end{eqnarray*}

if $p>1$ is small enough 
(using the fact that $\sum_x e^{-\tilde a_0\frac{|x|^2}{ps}}=O(\min(s,r^2))$).
\item Now, the contribution of the term coming from the composition of the the dominating term of \refeq{A31} and of the error term $e'_2$
term of \refeq{A32} is in

$\sum_{\mathcal C,\mathcal C'}\sum_{k_1+r+\ell+s\le n}\sum_x
\frac{\bar\mu(\tilde{\mathcal C}')}{r}e^{-\frac{\langle(\Sigma^2)^{-1}
x,x\rangle}{2r}}k\frac{\bar\mu(\tilde E_{\ell,\mathcal C})^{\frac 1p}}
{s^{\frac 32}}=
$
\begin{eqnarray*}
&=&n^{\frac 1{100}}\log n\sum_{k_1+r+\ell+s\le n}\frac{
\min(r,s^2)}{r\ell^{\frac 1p} s^{\frac 32}}\\
&=&n^{\frac 1{100}}\log n\sum_{k_1+r+\ell+s\le n}\frac{
r^{\frac 34}(s^2)^{\frac 14}}{r\ell^{\frac 1p} s^{\frac 32}}\\
&=&n^{1+\frac 1{100}+\frac 34+1-\frac 1p}\log^2 n=o(n^2),
\end{eqnarray*}
if $p>1$ is small enough.
\item For the control of the sum over $(k_1,r,s,\ell)$ such that
$\min(r,s)<3k$, we proceed as we did for $A_2$.
\item It remains to estimate
$$\sum_{\mathcal C,\mathcal C'}\sum_{k_1+r+\ell+s\le n}(\bar\mu
(\tilde{\mathcal D}_{r+\ell+s,\mathcal C'}\cap T^{-r}(\tilde 
E_{\ell,\mathcal C}\cup \tilde 
{\mathcal D}_{\ell,\mathcal C})) +
\bar\mu
((\tilde{ E}_{r+\ell+s,\mathcal C'}\cup \tilde 
{\mathcal D}_{r+\ell+s,\mathcal C'})\cap T^{-r}\tilde 
{\mathcal D}_{\ell,\mathcal C}).$$
The dominating terms obtained by \refeq{eqprop4}
are estimated as the dominating terms of \refeq{A31} and\refeq{A32}.
They bring a contribution to $A_3$ in
$$\delta^k \sum_{k_1+r+\ell+s\le n}\frac 1{(r+s)\ell}\le
  \delta^kn^2\log n=o(n^2).$$
The fact that the other terms are in $o(n^2)$ follows
as for the study of \refeq{approx}.
\end{itemize}
\item \underline{Control of $A_4$}.
We have $A_4\le \sum_{1\le k_1<\ell\le n}\PP(E_{k_1,\ell})=O(n\log n)
=o(n^2)$.
\end{itemize}
Finally we have $\var _{\bar\mu}(V_n)\sim 8(A_2+ A_3)$.
\end{proof}
\begin{lem}\label{A20}
We have
$$\sum_{k_1\ge 1,r\ge 0,\ell\ge 1,s\ge 0:k_1+r+\ell+s\le n}\frac 1{(r+\ell)(\ell+s)}\sim\frac{\pi^2}{12}n^2.$$
\end{lem}
\begin{proof}
Comparing the sum with an integral (by the Lebesgue
dominated convergence theorem) and making the change of variables
$r=\min(u+v,u+w)$ and $s=\max(u+v,u+w)$, we obtain
\begin{eqnarray*}
\sum_{r+\ell+s\le n}\frac 1{(r+\ell)(\ell+s)}&\sim&n\int_{\{u,v,w>0\ :\ u+v+w\le 1\}}
   \frac{du\, dv\, dw}{(u+v)(u+w)}\\
&\sim&2n\int_0^1\left(\int_u^{\frac {1+u}2}\frac 1r\left(
   \int_r^{1-r+u}\frac{ds}s\right)\, dr\right)\, du\\
&\sim&2n\int_0^1\left(\int_u^{\frac {1+u}2}\frac 1r\log\left(
    \frac{1+u}r-1\right)\, dr\right)\, du.
\end{eqnarray*}
But
$$\int_u^{\frac {1+u}2}\frac 1r\log\left(
    \frac{1+u}r-1\right)\, dr=\int_2^{1+\frac 1u}\frac{\log(w-1)}w\, 
dw=Re\left(Li_2(2)-Li_2\left(1-\frac 1u\right)\right), $$
with $Li_2$ the dilogarithm function.
Indeed, we recall that for $z\ge 1$, $Li_2(z)=
\frac{\pi^2}6-\int_1^z\frac{\log(t-1)}t\, dt-i\pi\log z$.
Recall that $Re(Li_2(2))=\frac {\pi^2}4$.
Using an explicit primitive of $u\mapsto Li_2(1+\frac 1 u)$
(such as $zLi_2(1-z^{-1})+Li_2(-z)+(\log z-i\pi)\log (z+1)$),
we find
that $Re\int_0^1 Li_2\left(1-\frac 1u\right)\, du=\frac{\pi^2}6$.
Hence $\sum_{r+\ell+s\le n}\frac 1{(r+\ell)(\ell+s)}\sim
2n\pi^2((1/4)-(1/6))=\pi^2n/6$ and so
$$\sum_{k_1+r+\ell+s\le n}\frac 1{(r+\ell)(\ell+s)}
=\sum_{k_1=1}^{n-1} \sum_{r+\ell+s\le n-k_1}\frac 1{(r+\ell)(\ell+s)}
=\sum_{k_1=1}^{n-1} \sum_{r+\ell+s\le k_1}\frac 1{(r+\ell)(\ell+s)}\sim  \frac{\pi^2n^2}{12}.$$
\end{proof}
\begin{lem}\label{controlA2}
We have
$$
\sum_{k_1,r,\ell,s\ge 1:k_1+r+\ell+s\le n}\frac 1{r\ell+rs+s\ell}\sim  n^2 J,$$
with
$$J:=
\int_{[0,1]^3}\frac{(1-(u+v+w)){\bf 1}_{\{u+v+w\le 1\}}\,  du\, dv\, dw}
    {uv+uw+vw}.$$
\end{lem}
\begin{proof}
We have
\begin{eqnarray*}
\sum_{k_1+r+\ell+s\le n}\frac{1}{r\ell+rs+s\ell} 
&=& \sum_{r+\ell+s\le n}\frac{n-(r+\ell+s)}{r\ell+rs+s\ell} \\
&=& n^2\int_{[0,1]^3}f\left(\frac{\lceil nu\rceil}n,
\frac{\lceil nv\rceil}n,\frac{\lceil nw\rceil}n\right)\, dudvdw\\
&\sim&n^2\int_{[0,1]^3}f(u,v,w)\, dudvdw=n^2J,
\end{eqnarray*}
with $f(u,v,w):=\frac{1-u-v-w}{uv+uw+vw}{\mathbf 1}_{\{
       u+v+w\le 1\}}$,
due to the Lebesgue dominated convergence theorem.
\end{proof}
\section{Proof of Theorem \ref{thm1}}\label{expectation2}
\begin{coro}\label{coro2}
Let $P$ be a probability measure on $\bar M$ with density $h$
with respect to $\bar\mu$. Assume that $h$ is in ${\mathbb L}^2(\bar\mu)$.
Then $${\mathbb E}_P[V_n]=cn\log n+O(n).$$
\end{coro}
\begin{proof}[Proof of Corollary \ref{coro2}]
We have
\begin{eqnarray*}
|{\mathbb E}_P[V_n]-{\mathbb E}_{\bar\mu}[V_n]|&=&{\mathbb E}_{\bar\mu}
[(V_n-{\mathbb E}_{\bar\mu}[V_n])h]\\
&\le& \sqrt{\var _{\bar\mu}(V_n)}\Vert h\Vert_2=O(n)\Vert h\Vert_2=O(n),
\end{eqnarray*}
according to Theorem \ref{prop1}. We conclude thanks to
Theorem \ref{thm1a}.
\end{proof}
For any $t>0$, we define $n_t$ on $\mathcal M$ by
$n_t:=\max\{m\ge 0:\sum_{k=0}^{m-1}\tau\circ T^k\le t\}$
the number of reflections before time $t$.
\begin{coro}\label{coro2b}
Let $h$ be a probability density with respect to $\bar\mu$
belonging to ${\mathbb L}^p(\bar\mu)$ for some $p>2$. We have
$${\mathbb E}_{h\bar\mu}[V_{n_t}]= ct\log t/{\mathbb E_{\bar\mu}[\tau]}
+O(t), \ \ \mbox{as }t\mbox{ goes to infinity.}$$

\end{coro}
\begin{proof}
To simplify notations, we write $\bar\tau:={\mathbb E}_{\bar\mu}[\tau]$.
Observe that $n_t\le t/\min\tau$ on $\bar M$.
We define
$$D:=\left|\sum_{k,j=0}^{n_t-1}{\mathbf 1}_{E_{k,j}}
-\sum_{k,j=0}^{\lfloor t/\tau\rfloor-1}{\mathbf 1}_{E_{k,j}}\right|
    \le 2\sum_{k=0}^{\lfloor t/\min\tau\rfloor-1}
  \sum_{j=\min(n_t,\lfloor t/\bar\tau\rfloor)}^
     {\max(n_t,\lfloor t/\bar\tau\rfloor)-1}{\mathbf 1}_{E_{k,j}}. $$
Due to corollary \ref{coro2}, it is enough to prove that
${\mathbb E}_{h\bar\mu}[D]=O(t)$.
Recall that (see \cite{FPCMP})
\begin{equation}\label{controlent}
\forall m\ge 1,\ \ \exists \tilde K_m,\ \ \sup_{t>0}
  \left\Vert  n_t-\frac t{\bar\tau}
  \right\Vert_{m}^m\le\tilde K_m t^{\frac m2}.
\end{equation}
Let $\varepsilon>0$. 
Due to Proposition \ref{prop2}, for some $C>0$, we have
\begin{eqnarray*}
\displaystyle{\mathbb E}_{h\bar\mu}\left[D{\mathbf 1}_{|n_t-(t/\bar\tau)|\le
   \varepsilon t}\right]
&\le& 2\sum_{k=0}^{\lfloor \frac t{\min\tau}\rfloor-1}
\sum_{j=\lfloor \frac t{\bar\tau}-\varepsilon t\rfloor-1}^{
\lfloor \frac t{\bar\tau}+\varepsilon t\rfloor-1}{\mathbb E}_{\bar\mu}
[h{\mathbf 1}_{E_{k,j}}]\\
&\le& 2 \sum_{k=0}^{\lfloor \frac t{\min\tau}\rfloor-1}
\sum_{j=\lfloor\frac t{\bar\tau}-\varepsilon t\rfloor}
^{\lfloor\frac t{\bar\tau}+\varepsilon t\rfloor}\Vert h\Vert_p(\bar\mu(E_{k,j}))^{1-\frac 1p}\\
&\le& 2\Vert h\Vert_p\left(\lceil 2\varepsilon t\rceil
  +2\sum_{j=0}^{\lceil 2\varepsilon t\rceil}
  \sum_{r=1}^{\lfloor\frac t{\min\tau}\rfloor}
 Cr^{\frac 1p -1}\right)=O(\varepsilon t^{1+\frac 1p}).
\end{eqnarray*}
Moreover, for any $m\ge 1$, we have
\begin{eqnarray*}
{\mathbb E}_{h\bar\mu}\left[D{\mathbf 1}_{|n_t-(t/\bar\tau)|>
   \varepsilon t}\right]
    &\le&{\mathbb P}(|n_t-(t/\bar\tau)|> \varepsilon t|)^{\frac 12-\frac 1p}\Vert h\Vert_p\left({\mathbb E}_{\bar\mu}
  \left[\left(\sum_{k,j=0}^{\lfloor\frac t{\min\tau}\rfloor-1}{\mathbf 1}
       _{E_{k,j}}\right)^2\right]\right)^{\frac 12}\\
&\le& \left(\frac {\tilde K_m t^{\frac m2}}{(\varepsilon t)^m}\right)^{\frac 12-\frac 1p}
\Vert h\Vert_p\left(\var _{\bar\mu}(V_{\lfloor\frac t{\min\tau}\rfloor})
      +({\mathbb E}_{\bar\mu}[V_{\lfloor\frac t{\min\tau}\rfloor}])^2\right)^{\frac 12}\\
&\le& \left(\frac {\tilde K_m}{\varepsilon^m t^{\frac m2}}\right)
^{\frac 12-\frac 1p} \Vert h\Vert_p  Ct\log t=O\left(t\log t
    \varepsilon^{-\tilde m}t^{-\frac {\tilde m} 2}\right),
\end{eqnarray*}
with $\tilde m:=m(p-2)/2p$ (due to \refeq{controlent}, to Theorem \ref{prop1} and to Theorem \ref{thm1a}).
Take $\varepsilon=t^{\frac {-(1/p)-(\tilde m/2)}{\tilde m+1}}$.
We obtain
$${\mathbb E}_{h\bar\mu}[D]=O(\varepsilon t^{1+\frac 1p}\log t)
=o(t),$$
by taking $m$ large enough since $p>2$.
\end{proof}
\begin{coro}\label{coro2c}
Let $H$ be a probability density with respect to $\nu$ on $\mathcal M$
such that
$$h:(q,\vec v)\mapsto \sum_{\ell\in\mathbb Z^2}
   \int_0^{\tau(q,\vec v)}H(q+\ell+s\vec v,\vec v)\,ds $$
belongs to ${\mathbb L}^p(\bar\mu)$ for some $p>2$.
Then
$${\mathbb E}_{H\nu}[\mathcal V_t]= ct\log t/{\mathbb E_{\bar\mu}
[\tau]}+O(t),\ \ \ \mbox{as }t\mbox{ goes to infinity}.$$
\end{coro}
\begin{proof}[Proof of Theorem \ref{thm1}]
For every 
$(q,\vec v)\in\bar M$, every $\ell\in\mathbb Z^2$ and every
$s\in[0,\tau(q,\vec v))$, we have
$$\mathcal V_t(q+\ell+s\vec v,\vec v)=O(n_t)+
V_{n_t(q,\vec v)}(q,\vec v)=O(t)+V_{n_t(q,\vec v)}(q,\vec v).  $$
So, due to \refeq{suspension}, 
we have
$${\mathbb E}_{H\nu}\left[\mathcal V_t
\right]=O(t)+{\mathbb E}_{h\mu_{|\bar M}}[V_{n_t}]
={\mathbb E}_{2h\sum_i|\partial O_i|\bar\mu}\left[V_{n_t}
\right]= ct\log t/{\mathbb E_{\bar\mu}[\tau]}+O(t),$$
according to Corollary \ref{coro2b}.
\end{proof}
\begin{proof}[Proof of Theorem \ref{thm1}]
We apply directly Corollary \ref{coro2c} with
$H(q,\vec v)={\mathbf 1}_{q\in[0,1]^2}$ and $h(q,\vec v)=\tau(q,\vec v)$.
\end{proof}
\section{Almost sure convergence}\label{LFGN}
In this section we prove Corollary \ref{thm2}
by a classical argument (see \cite{DE} for example).
Let $\gamma\in(0,1/2)$. 
According to Theorems \ref{thm1a} and \ref{prop1}, we have
$$\var_{\bar\mu}( V_n)/({\mathbb E}_{\bar\mu}[V_n])^2 =O(\log^{-2}n).$$
Due to the Bienaym\'e-Chebychev inequality and to
the Borel Cantelli lemma, this implies the $\bar\mu$-almost sure convergence 
of $(V_{\exp n^{\frac {1+\gamma}2}}/{\mathbb E}_{\bar\mu}
[V_{\exp n^{\frac {1+\gamma}2}}])_n$ to 1.
Therefore, the following convergence holds $\bar\mu$-almost surely:
$$\lim_{n\rightarrow+\infty}
\frac{V_{\exp n^{\frac {1+\gamma}2}} }
{n^{\frac{1+\gamma}2} \exp n^{\frac {1+\gamma}2} }=c.$$
Now, for every integer $N\in[\exp n^{\frac {1+\gamma}2}, \exp (n+1)^{\frac {1+\gamma}2}]$, we have
$$ \frac{V_{\exp n^{\frac {1+\gamma}2}}}{(n+1)^{\frac{1+\gamma}2} \exp (n+1)^{\frac {1+\gamma}2} } \le \frac{V_N}{N\log N}\le \frac{V_{\exp (n+1)^{\frac {1+\gamma}2}}}     {n^{\frac{1+\gamma}2} \exp n^{\frac {1+\gamma}2} }.$$
Since $\lim_{n\rightarrow+\infty}{(n+1)^{\frac{1+\gamma}2} \exp (n+1)^{\frac {1+\gamma}2}}/({ n^{\frac {1+\gamma}2}\exp n^{\frac {1+\gamma}2} })=1$, we conclude the $\bar\mu$-almost sure convergence of $(V_{n}/(n\log n))_n$ to $c$.

For any $t>0$, we write $n_t$ for the number of reflection times before
time $t$. Recall that $(t/n_t)_t$ converges $\bar\mu$-almost surely to 
$\mathbb E_{\bar\mu}[\tau]$ as $t$ goes to infinity. 
Hence we have, $\bar\mu$-almost surely,
$$\frac{V_{n_t}}{t\log t}
\sim \frac{V_{n_t}}{\mathbb E_{\bar\mu}[\tau]n_t\log n_t}
\sim \frac c{\mathbb E_{\bar\mu}[\tau]},\ \ \mbox{ as }t\rightarrow+\infty.$$
Since $V_{n_t}(q+\ell,\vec v)=V_{n_t}(q,\vec v)$ for every
$(q,\vec v)\in\bar M$ and every $\ell\in\mathbb Z^2$.
We also have, $\mu$-almost everywhere, $\frac{V_{n_t}}{t\log t}
\sim \frac c{\mathbb E_{\bar\mu}[\tau]}$.
Recall now that
$$\forall(q,\vec v)\in M,\ \forall s\in[0,\tau(q,\vec v)),\ \ \ \  \left|\mathcal V_t(q+s\vec v,\vec v)-V_{n_t(q,\vec v)}(q,\vec v)
\right|\le 2n_t\le 2\frac t{\min\tau}.  $$
Hence, due to \refeq{suspension}, we obtain
$$\nu\left(\left\{\frac{\mathcal V_t}{t\log t}\not\rightarrow 
 \frac c{\mathbb E_{\bar\mu}[\tau]}\right\}\right)
\le(\max\tau) \mu\left(\left\{\frac{ V_{n_t}}
     {t\log t}\not\rightarrow 
 \frac c{\mathbb E_{\bar\mu}[\tau]}\right\}\right) =0.$$
\qed

\end{document}